\RequirePackage{ifpdf}
\ifpdf 
\documentclass[pdftex]{mfat}
\else
\documentclass{mfat}
\fi
\usepackage[all,2cell]{xy}  \UseAllTwocells \SilentMatrices
\usepackage{adjustbox}

\newtheorem{lem}{Lemma}[section]

 { \theoremstyle{definition}

 }

\begin{document}	

\title[Geometric Regularity Results on $B_{\alpha,\beta}^{k}$-Manifolds,
I: Affine Connections] 
	  {Geometric Regularity Results on $B_{\alpha,\beta}^{k}$-Manifolds,
I: Affine Connections} 

\author[Yuri Ximenes Martins]{Yuri X. Martins}
\email[Yuri Ximenes Martins]{\href{mailto:yurixm@ufmg.br}{yurixm@ufmg.br}}
\address{Instituto de Ci\^encias Exatas, Universidade Federal de Minas Gerais, Belo Horizonte, Brazil}
\thanks{Y. X. M is supported by CAPES, grant number 88887.187703/2018-00.}

\author[Rodney Josu\'e Biezuner]{Rodney J. Biezuner}
\email[Rodney Josu\'e Biezuner]{\href{mailto:rodneyjb@ufmg.br}{rodneyjb@ufmg.br}}
\address{Instituto de Ci\^encias Exatas, Universidade Federal de Minas Gerais, Belo Horizonte, Brazil}

\subjclass[2020]{58A05, 53C05, 46M15} 
\date{dd/mm/yyyy; Revised dd/mm/yyyy}
\keywords{Affine Connections, Regularity Structures, Presheaves of Fr\'echet Spaces, Existence and Multiplicity}

\begin{abstract}
	In this paper we consider existence and multiplicity results concerning affine connections on $C^{k}$-manifolds $M$ whose coefficients are as regular as one needs, following the regularity theory introduced in \cite{eu_1}.  We show that if $M$ admits a $B_{\alpha,\beta}^{k}$-structure, then the existence of such regular connections can be established in terms of properties of the structural presheaf $B$. In other words, we propose a solution to the existence problem in this setting. With regard to the multiplicity problem, we show that the space of regular affine connections is an affine space of the space of regular $\operatorname{End}(TM)$-valued 1-forms, and that if two regular connections are locally additively different, then they are actually locally different. The existence of a topology in which the space of regular connections is a nonempty open dense subset of the space of all regular $\operatorname{End}(TM)$-valued 1-forms is suggested.
\end{abstract}

\maketitle
\thispagestyle{fancy}

\section{Introduction \label{sec_introduction}}
In a naive sense, a ``regular'' mathematical object in one which is described by functions. The space where these structural functions live is the ``regularity'' of the defined mathematical object. Thus, $n$-manifolds $M$ are ``regular'' because they are described by its charts $\varphi_i:U_i\rightarrow \mathbb{R}^n$.  Furthermore, its regularity is given by the space where the transition functions $\varphi_{ji} =\varphi_j\circ \varphi^{-1}_i$ live. Similarly, tensors $T$, affine connections $\nabla$ and partial differential operators $D$ in $M$ are also ``regular'', because they are described by their local coefficients.

In \cite{eu_1} the authors proposed a formalization to this general notion of ``regularity'' and begin its study in the context of $C^k$-manifolds. In particular, we considered the existence problem of subatlases satisfying prescribed regularity conditions. In the present work we will continue this study by considering existence and multiplicity problems for regular affine connections on regular $C^k$-manifolds.

Such problems are of wide interest in many areas. For instance, in gauge theory and when one needs certain
uniform bounds on the curvature, it is desirable to work with connections
$\nabla$ whose coefficients are not only $C^{k-2}$, but actually
$L^{p}$-integrable or even uniformly bounded \cite{gauge,bounded_geometry_1}.
In other situations, as in the study of holomorphic geometry, one
considers holomorphic or Hermitian connections over complex manifolds 
whose coefficients are holomorphic \cite{holomorphic_1,atiyah_class}.
If $M$ has a $G$-structure, one can also consider $G$-principal
connections $\omega$ on the frame bundle $FM$, which induce an affine
connection $\nabla_{\omega}$ in $M$, whose coefficients $\omega_{b}^{c}(\partial_{a})=\Gamma_{ab}^{c}$
are such that for every $a$, $\Gamma_{a}\in C^{k-2}(U)\otimes\mathfrak{g}$,
where $(\Gamma_{a})_{cb}=\Gamma_{ab}^{c}$ \cite{kobayashi}. Interesting
examples of these are the symplectic connections, which are specially
important in formal deformation quantization of symplectic manifolds
\cite{symplectic_connection,deformation_quantization}. As a final
example, a torsion-free affine connection defines a $L_{3}$-space
structure in a smooth manifold $M$ (in the sense of \cite{L3_space})
iff its coefficients satisfy a certain system of partial differential
equations \cite{L3_space_2}. Thus, the theory of $L_{3}$-spaces is about certain regular connections.

In the formulation of ``regular'' object of \cite{eu_1}, the spaces describing the regularity are modeled by nuclear Fr\'echet spaces, since their category $\mathbf{NFre}$ is a closed symmetric monoidal relatively to the projective tensor project $\otimes$, and since most of the function spaces arising from geometry belongs to $\mathbf{NFre}$.   Furthermore, since one such space for every open set $U\subset M$, this leads to consider presheaves $B:\operatorname{Open}(\mathbb{R}^n)^{op}\rightarrow \mathbf{NFre}$. On the other hand, notice that if a mathematical concept is described by $C^k$-functions $f$, then the regularity of this object is typically described not only by the space where the functions $f$ live, but also by the space where their derivatives $\partial^i f$, $1\leq i\leq k$, belong. Thus, we should consider sequences of presheaves $B_i:\operatorname{Open}(\mathbb{R}^n)^{op}\rightarrow \mathbf{NFre}$, with $i=0,...,k$. But, for reasons which will be more clear later (see Remark \ref{important_remark}), it is better to allow sequences $B_i$ indexed in more general sets $\Gamma$. These are called \textit{presheaves of $\Gamma$-spaces}. 
\begin{example}
The prototypical example of $[0,k]$-presheaf, where $[0,k]={0,...,k}$, is the sequence $C^{k-i}$. If $\beta:\Gamma \rightarrow [0,k]$ is any function we get the $\Gamma$-presheaf $C^{k-\beta(i)}$. One can also consider the sequences $L^i(U)$, $W^{i,p}(U)$, and so on.
\end{example}
We would also like to sum and multiply in the regularity spaces in a distributive way. We will consider presheaves of \textit{distributive $\Gamma$-spaces}, which are presehaves of $\Gamma$-spaces endowed with a distributive structure. On the other hand, recall that we will be interested in regularity structures on $C^k$-objets. Thus, we need some way to make sense of the intersections $B_{\alpha(i)}(U)\cap C^{k-\beta(i)}(U)$, where the index $\alpha$ is necessary to ensure that both are presheaves of $\Gamma$-spaces for the same $\Gamma$. This will be done with the help of an ambient category $\mathbb{X}$ with pullbacks (typically the category of $\Gamma$-graded real vector spaces) and of an ambient object $X$, in which both $B_i$ and $C^{k-\beta(i)}$ can be included, being the intersection $B_i\cap_X C^{k-\beta(i)}$ given by the pullbacks of the inclusions. In other words, we will need an \textit{intersection structure} between $B_i$ and $C^{k-\beta(i)}$, denoted by $\mathbb{X}$.

With all this stuff, we define a ``regular'' function $f:U\rightarrow \mathbb{R}$, here called \textit{$(B,k,\alpha,\beta)$-function in $\mathbb{X}$}, as a $C^k$-function such that $\partial^i f\in B_{\alpha(i)}(U)\cap_{X(U)}C^{k-\beta(i)}(U)$ for each $i$. Similarly, a ``regular'' $C^k$-manifold, called \textit{$(B^k_{\alpha,\beta},\mathbb{X})$-manifold} is one whose transition functions $\varphi_{ji}$ are regular functions, i.e., are $(B,k,\alpha,\beta)$-functions in $\mathbb{X}$. Finally, a ``regular'' affine connection in a $(B^k_{\alpha,\beta},\mathbb{X})$-manifold $M$ is a $C^k$ affine connection $\nabla$ whose coefficients $(\Gamma_\varphi)^{c}_{ab}$ in each regular chart $\varphi$ of $M$ are $(B,k,\alpha',\beta')$-functions in some other ambient $\mathbb{Y}$. We also say that $\nabla$ is an affine $(B^k_{\alpha',\beta'},\mathbb{Y})$-connection in $M$. 

In this new language, the existence and multiplicity problems for regular affine connections on regular manifolds have the following description:

\begin{itemize}
    \item \textbf{Existence.} Let $M$ be a $(B^k_{\alpha,\beta},\mathbb{X})$-manifolds. Given $\alpha'$ and $\beta'$ and $\mathbb{Y}$, can we find conditions on $B$, $k$, $\alpha$ and $\beta$ such that $M$ admits affine $(B^k_{\alpha',\beta'},\mathbb{Y})$-connections? 
    \item \textbf{Multiplicity.} What can be said about the dimension of the space $\operatorname{Conn}^k_{\alpha',\beta'}(M;\mathbb{Y)}$ of $(B^k_{\alpha',\beta'},\mathbb{Y})$-connections in $M$?
\end{itemize}

 We prove that arbitrarily regular connection exist in sufficiently regular manifolds. The will be based on four steps: 
\begin{enumerate}
    \item existence of arbitrarily regular affine connections in each $U\subset M$ (Proposition \ref{prop_local_existence});
    \item existence of global affine connections in $M$ whose coefficients $(\Gamma_\varphi)^c_{ab}$ are regular, but whose derivatives are not (Theorem \ref{thm_existence_connections});
    \item proving a Regularity Globalization Lemma, ensuring that the derivatives $\partial^i(\Gamma_\varphi)^c_{ab}$ can be made regular (Lemma \ref{key_lemma});
    \item putting together the previous steps to get a global regular connection (Theorem \ref{corollary_existence_connections}).
\end{enumerate}
 
 Concerning the multiplicity problem, in Proposition \ref{space_affine_connections_is_affine}  we show that
  the space of regular affine connections is an affine space of the space of regular $\operatorname{End}(TM)$-valued 1-forms, and in Theorem \ref{multiplicity_thm} we prove that if two regular connections are locally additively different, then they are locally different.

The paper is organized as follows. In Section \ref{sec_background}
we present the needed background for our main results. In parts, it contain concepts introduced in \cite{eu_1}, but now described in a more concrete way. Subsection  \ref{sec_globalization_lemma} and Subsection \ref{sec_globalization_lemma} are fully new in content if compared to \cite{eu_1}. In particular, in Subsection \ref{sec_globalization_lemma} the Regularity Globalization Lemma is stated and proved. Section \ref{sec_a_b_connections} deals with the existence problem and the four steps described above. Section \ref{sec_multiplicity} is about the multiplicity problem. The paper ends with an informal discussion of why we should believe that there exists a topology in which $\operatorname{Conn}^k_{\alpha',\beta'}(M;\mathbb{Y)}$ is a nonempty open dense subset of the space of tuples of $(B,k,\alpha',\beta')$-functions in $\mathbb{Y}$. 

\begin{itemize}
    \item \textbf{Convention}: remarks concerning notations and assumptions will be presented using a bullet.
\end{itemize} 

\section{Background \label{sec_background}}
In this section we will present all that is necessary to formally state and prove the existence and multiplicity results described in the introduction. We begin by recalling in detail the definitions and results from \cite{eu_1} that will be needed in the next sections, which are basically the notion of $B^k_{\alpha,\beta}$-manifold and the underlying concepts.
\begin{itemize}
    \item \textbf{Notations}. In the following $\Gamma$ denotes an arbitrary set, regarded as a set of indexes. Given $k,l\geq0$ with $k\leq l$,  $[k,l]$ will denote the integer closed interval from $k$ to $l$. $\mathbf{NFre}$ will denote the category of nuclear Fr\'echet spaces with continuous linear maps, and $\otimes$ is the projective tensor product, which makes $\mathbf{NFre}$ a symmetric monoidal category \cite{nuclear}.
\end{itemize}

\subsection{$\Gamma$-Spaces}
In this section we will recall what are distributive $\Gamma$-spaces. For a complete exposition, see Section 2 of \cite{eu_1}.
\begin{definition}
A \emph{$\Gamma$-space} is a family  $B=(B_i)_{i\in\Gamma}$ of  nuclear Fr\'echet spaces. A \textit{morphism} $f:B\Rightarrow B'$ between two $\Gamma$-spaces $B$ and $B'$ is a family $f_i:B_i\rightarrow B'_i$ of morphisms in $\mathbf{NFre}$,  with $i\in\Gamma$. Composition and identities are defined componentwise, so that we have a category $\mathbf{NFre}_{\Gamma}$. 
\end{definition}
\begin{example}\label{example_1_gamma_spaces}
The fundamental examples are the $[0,k]$-spaces given by $B_i=C^i(M;\mathbb{R}^m)$, where $M$ is a $C^k$-manifold, with the standard family of semi-norms
\begin{equation}\label{frechet_ck}
 \Vert f\Vert_{r,l}=\sum_{j}\sup_{\vert\mu\vert=r}\sup_{x\in K_{l}}\vert\partial^{\mu}f_{j}(x)\vert   
\end{equation}
with $0\leq r\leq i$ and $(K_{l})$ some nice sequence of compact
sets of $M$. More generally, given a set $\Gamma$ and a function $\beta:\Gamma\rightarrow [0,k]$, we can define the $\Gamma$-space $B_i=C^{k-\beta(i)}(M;\mathbb{R}^m)$.
\end{example}
\begin{example}
Similarly, we can consider in $C^{k-\beta(i)}(U;\mathbb{R})$ other Banach norms, as $L^p$-norms or Sobolev norms.
\end{example}
\begin{example}
More generally, any sequence of Banach spaces defines a $\Gamma$-space, where $\Gamma$ is the set in which the sequence is indexed.
\end{example}
\begin{example}\label{example_2_gamma_spaces}
Given a $\Gamma$-space $B$ and a function $\alpha:\Gamma' \rightarrow \Gamma$ we define a new $\Gamma'$-space $B_{\alpha}$ by reindexing $B$, i.e., $(B_\alpha)_i=B_{\alpha(i)}$.
\end{example}
\begin{definition}\label{definition_distributive_space}
A \emph{distributive structure} in a $\Gamma$-space $B$ consists of: 
\begin{enumerate}
    \item maps $\epsilon,\delta:\Gamma\times\Gamma \rightarrow\Gamma$;
    \item continuous linear maps $*_{ij}:B_{i}\otimes B_{j}\rightarrow B_{\epsilon(i,j)}$
and $+_{ij}:B_{i}\otimes B_{j}\rightarrow B_{\delta(i,j)}$, for every $i,j\in \Gamma$
\end{enumerate}
such that the following compatibility equations are satisfied for every $x\in B_{i}$, $y\in B_{j}$
and $z\in B_{k}$, and every $i,j,k\in\Gamma$ (notice that the first two equations, which describe left and right distributivity, respectively,  makes sense only because
of the last two).
\begin{eqnarray*}
x*_{i\delta(j,k)}(y+_{jk}z) & = & (x*_{ij}y)+(x*_{ik}z)\\
(x+_{ij}y)*_{\delta(i,j)k}z & = & (x*_{ik}z)+(y*_{jk}z)\\
\epsilon(i,\delta(j,k)) & = & \delta(\epsilon(i,j),\epsilon(i,k))\\
\epsilon(\delta(i,j),k) & = & \delta(\epsilon(i,k),\epsilon(j,k)).
\end{eqnarray*}
We also require that $+_{ii}$ coincides with the sum $+_i$ in $B_{i}$.
\end{definition}
\begin{definition}
A \emph{distributive $\Gamma$-space} is a $\Gamma$-space $B$ endowed with distributive structure $(\epsilon,\delta,+,*)$. A  \textit{morphism} between distributive $\Gamma$-spaces $(B,\epsilon,\delta,+,*)$ and $(B',\epsilon',\delta',+',*')$  is a pair $(f,\mu)$, where $\mu:\Gamma \rightarrow \Gamma$ is a function and $f=(f_i)_{i\in \Gamma}$ is a family of continuous linear maps $f_i:B_i\rightarrow B'_{\mu(i)}$ such that the following equations are satisfied (notice again that the first two equations, which describe left and right distributivity, respectively,  makes sense only because
of the last two):
\begin{eqnarray*}
f(x*_{ij}y)&=&f(x)*'_{\epsilon'(\mu(i),\mu(j))}f(y) \\ 
f(x+_{ij}y)&=&f(x)+'_{\delta'(\mu(i),\mu(j))}f(y) \\
\mu(\epsilon(i,j))&=& \epsilon'(\mu(i),\mu(j))\\
\mu(\delta(i,j))&=& \delta'(\mu(i),\mu(j)).
\end{eqnarray*}
\end{definition}
\begin{example}\label{example_distributive_gamma_spaces}
With pointwise sum and multiplication, the $\Gamma$-spaces $C^{k-\beta(i)}(M;\mathbb{R})$ are distributive for $\delta(i,j)=k-\max(\beta(i),\beta(j))=\epsilon(i,j)$. H\"older's inequality implies that the $\mathbb{Z}_{\geq 0}$-spaces $L^i(U)$ also are distributive, with $\delta(i,j)=\min(i,j)$ and $\epsilon(i,j)=i\star j= (i\cdot j)/(i+j)$, now viewed as a $\Gamma$-space for $\Gamma\subset \mathbb{Z}_{\geq 0}$ being the set of all $i,j\geq 0$ such that $i\star j$ is an integer. From Yong's inequality one can considers a different distributive structure in $L^i(\mathbb{R}^n)$ given by convolution product and such that $\epsilon(i,j)= r(i,j)=(i\cdot j)/(i+j-(i\cdot j))$, where now $\Gamma \subset \mathbb{Z}_{\geq 0}$ is such that $r(i,j)$ is an integer.
\end{example}
\begin{definition}
Let $B$ and $B'$ be a $\Gamma$-space and a $\Gamma'$-space, respectively. The \textit{external tensor product} between them is the $\Gamma\times\Gamma'$-space $B\otimes B'$ such that $(B\otimes B')_{i,j}=B_i\otimes B'_j$. 
\end{definition}
\begin{remark}\label{example_B_epsilon}
A distributive $\Gamma$-space $(B,\epsilon,\delta,+,*)$ induces two $(\Gamma\times\Gamma)$-spaces $B_\epsilon$ and $B_\delta$, defined by $(B_{\epsilon})_{ij}=B_{\varepsilon(i,j)}$ and $(B_{\delta})_{ij}=B_{\delta(i,j)}$. The family of maps $*_{ij}$ and $+_{ij}$ defining a distributive $\Gamma$-structure are actually morphisms of $\Gamma\times\Gamma$-spaces $*:B\otimes B\Rightarrow B_{\epsilon}$ and $+:B\otimes B\Rightarrow B_{\delta}$. 
\end{remark}

\subsection{Abstract Intersections}

 We now recall abstract intersection structures between presheaves of distributive $\Gamma$-spaces. For a complete exposition, see Section 2 of \cite{eu_1}.
 
As a motivation, recall that the intersection $A\cap B$ of vector subspaces $A,B\subset X$ of a fixed vector space $X$  can be viewed as the pullback between the inclusions $A\hookrightarrow X$ and $B \hookrightarrow X$. However, if $A$ and $B$ are arbitrary vector spaces, the ``abstract intersection'' $A\subset B$ is a priori not well-defined. In order to define it we first consider an ambient space $X$ in which both can be embedded and one fixes embeddings $\alpha:A\hookrightarrow X$ and $\beta:B\hookrightarrow X$. The \textit{abstract intersection} between $A$ and $B$ inside $X$, relative to the fixed embeddings $\alpha$ and $\beta$, is then defined  as the pullback $\operatorname{pb}(\alpha,\beta)$ between $\alpha$ and $\beta$.

Notice that if the category $\mathbf{C}$ in which the objects $A$ and $B$ live does not have pullbacks, then the above notion of abstract intersection may not exist for certain triples $(X,\alpha,\beta)$. In these cases we first need to regard both $A$ and $B$ as objects of another category $\mathbf{X}$ with pullbacks. Thus, we take a faithful functor $\gamma:\mathbf{A}\rightarrow \mathbf{X}$, a fixed ambient object $X\in \mathbf{X}$ and fixed monomorphisms $\alpha:\gamma_{\mathbf{A}}(A)\hookrightarrow X$ and $\alpha:\gamma(A)\hookrightarrow X$ $\alpha:\gamma(B)\hookrightarrow X$. The abstract intersection can then be defined as the pullback $\operatorname{pb}(\alpha,\beta)$ between $\alpha$ and $\beta$, but now it depends additionally on the faithful functor $\gamma$. Furthermore, a priori it exists only as an object of $\mathbf{X}$, which is why we call it ``abstract''.

\begin{definition}
A \emph{$\Gamma$-ambient} is a tuple $\mathcal{X}=(\mathbf{X},\gamma,\gamma_{\Sigma})$,
where $\mathbf{X}$ is a category with pullbacks and $\gamma$ and
$\gamma_{\Sigma}$ are faithful functors $\gamma:\mathbf{NFre}_{\Gamma}\rightarrow\mathbf{X}$
and $\gamma_{\Sigma}:\mathbf{NFre}_{\Gamma\times\Gamma}\rightarrow\mathbf{X}$. We say that a $\Gamma$-ambient $\mathcal{X}$ \textit{has null objects} if the ambient category $\mathbf{X}$ has a null object $0$.
\end{definition}
\begin{example}\label{standard_ambient}
As discussed in \cite{eu_1}, the typical example of $\Gamma$-ambient is such that $\mathbb{X}$ is the category $\mathbf{Vec}_{\mathbb{R},\Gamma\times \Gamma}$ of $\Gamma\times \Gamma$-graded real vector spaces, $\gamma_\Sigma: \mathbf{NFre}_{\Gamma\times\Gamma}\rightarrow \mathbf{Vec}_{\mathbb{R},\Gamma\times \Gamma}$ is the forgetful functor and $\gamma:\mathbf{NFre}_{\Gamma}\rightarrow \mathbf{Vec}_{\mathbb{R},\Gamma\times \Gamma}$ is the composition of $\Gamma_\Sigma$ with the inclusion $\mathbf{NFre}_{\Gamma}\hookrightarrow \mathbf{NFre}_{\Gamma\times \Gamma}$ given by $\imath(B)_{i,j}=B_i\oplus 0$. 
\end{example}
\begin{example}\label{vectorial_ambient}
More generally, it is interesting to consider the class of \textit{vector $\Gamma$-ambients} \cite{eu_1}. These are such that $\mathbf{X}=\mathbf{Vec}_{\mathbb{R},\Gamma\times\Gamma}$, $\gamma_{\Sigma}$ creates null objects and $\gamma=\gamma_{\Sigma}\circ F$, where $F$ is another functor that creates null objects. 
\end{example}
\begin{definition}
Let $B,B'$ be two $\Gamma$-spaces and let $\mathcal{X}=(\mathbf{X},\gamma,\gamma_{\Sigma})$ be a $\Gamma$-ambient.  An \textit{intersection structure} between $B$ and $B'$ in $(\mathbf{X},\gamma,\gamma_{\Sigma})$ is given by an object $X\in\mathbf{X}$ and monomorphisms $\imath:\gamma(B)\hookrightarrow X$ and $\imath':\gamma(B')\hookrightarrow X$. We write $\mathbb{X}$ in order to denote $(X,\imath,\imath')$ and we say that $X$ is the \textit{ambient object} of $\mathbb{X}$.
\end{definition}

Notice that in the definition above the functor $\gamma_{\Sigma}$, which is part of the $\Gamma$-ambient, was not used yet. It will be used when taking distributive structures into account.
\begin{definition}
Let $(B,\epsilon,\delta,+,*)$ and $(B',\epsilon',\delta',+',*')$ be distributive $\Gamma$-spaces. An \textit{intersection structure} between them, denoted by $\mathbb{X}$, consists of\footnote{Recall the definitions of $B_\epsilon$ and $B_\delta$ in Remark \ref{example_B_epsilon}.} \begin{enumerate}
    \item a $\Gamma$-ambient $\mathcal{X}=(\mathbf{X}, \gamma,\gamma_{\Sigma})$;
    \item an intersection structure $\mathbb{X}_0=(X,\imath,\imath')$ between $B$ and $B'$;
    \item an intersection structure $\mathbb{X}_{*}=(X_{*},\jmath_{*},\jmath_{*}')$ between $B_\epsilon$ and $B'_{\epsilon'}$; 
    \item an intersection structure $\mathbb{X}_{+}=(X_{+},\jmath_{+},\jmath_{+}')$ between $B_\delta$ and $B'_{\delta'}$.
\end{enumerate}
The \textit{abstract intersection} between $(B,\epsilon,\delta,+,*)$ and $(B',\epsilon',\delta',+',*')$ in $\mathcal{X}$, relatively to $\mathbb{X}$, is:
\begin{enumerate}
    \item the abstract intersection $\operatorname{pb}(X,\imath,\imath';\mathcal{X},\mathbb{X}_0)$ between $B$ and $B'$;
    \item the abstract intersection $\operatorname{pb}(X_*,\jmath_{*},\jmath'_{*};\mathcal{X},\mathbb{X}_{*})$ between $B_{\epsilon}$ and $B'_{\epsilon'}$;
    \item the abstract intersection $\operatorname{pb}(X_+,\jmath_{+},\jmath'_{+};\mathcal{X},\mathbb{X}_+)$ between $B_{\delta}$ and $B'_{\delta'}$;
    \item the additional pullbacks between $\gamma_\Sigma(*) \circ\jmath_{*}$ and $\gamma_\Sigma(*') \circ\jmath'_{*}$, and between $\gamma_\Sigma(+) \circ\jmath_{+}$ and $\gamma_\Sigma(+') \circ\jmath'_{+}$, as below\footnote{We presented only the first of these pullbacks, since the second one is fully analogous.},  describing the abstract intersections between the multiplications $*$ and $*'$, and the sums $+$ and $+'$, respectively\footnote{Recall from Remark \ref{example_B_epsilon} that $*$ and $+$ can be viewed as morphisms defined on external tensor products.}, in $\mathcal{X}$ and relatively to $\mathbb{X}$. 
\end{enumerate}
\end{definition}
\begin{equation}\label{diagram_abstract_intersection}
    \xymatrix{\operatorname{pb}(\gamma_\Sigma(*) \circ\jmath_{*},\gamma_\Sigma(*') \circ\jmath'_{*}) \ar[rr] \ar[dd]  && \gamma _{\Sigma} (B' \otimes B') \ar[d]^{\gamma _{\Sigma}(*')} \\
 & \operatorname{pb}(\jmath_*,\jmath'_{*}) \ar[d] \ar[r]  & \gamma _{\Sigma}(B'_{\epsilon})   \ar@{^(->}[d]^{\jmath_{*} '} \\
  \gamma _{\Sigma} (B \otimes B) \ar[r]_{\gamma _{\Sigma} (*)} & \gamma _{\Sigma}(B_{\epsilon})  \ar@{^(->}[r]_-{\jmath_{*}}  & X_*}
\end{equation}
\begin{itemize}
    \item The abstract intersections between $*$ and $*'$, and between $+$ and $+'$, as above, will be denoted simply by $\operatorname{pb}(*,*';\mathcal{X},\mathbb{X}_*)$ and $\operatorname{pb}(+,+';\mathcal{X},\mathbb{X}_+)$, respectively.
\end{itemize}

\begin{example}
The abstract intersection usually depends strongly on the ambient object $X$. E.g, in the $\Gamma$-ambient of Example \ref{standard_ambient}, if $X=\gamma(B)\oplus\gamma(B')$, then $\operatorname{pb}(X,\imath,\imath';\mathcal{X},\mathbb{X})\simeq 0$ independently of $\imath$ and $\imath'$. On the other hand, if $X=\operatorname{span}(\gamma(B)\cup\gamma(B'))$, then the abstract intersection is a nontrivial graded vector space if $B_i\cap B'_i$ is nontrivial for some $i\in\Gamma$. See Examples 12-13 of \cite{eu_1}. 
\end{example}

\begin{remark}
In typical examples, if there is some $i\in \Gamma$ such that $B_i$ or $B'_i$ is a nontrivial vector space, then the abstract intersection between $B$ and $B'$ is nontrivial too. This depends on how we ``concretify'' the abstract intersections: see Lemma \ref{vectorial_intersection}.
\end{remark}

\subsection{Concretization}
Here we will recall the concretization procedures for abstract intersection structures. See Section 2 of \cite{eu_1}.

In the last section we described how we can intersect two objects $A,B\in \mathbf{C}$ in a category without pullbacks by means of considering spans in an ambient category $\mathbf{X}$ with pullbacks via a faithful functor $\gamma:\mathbf{C}\rightarrow \mathbf{X}$. The resulting abstract intersection, however, exists a priori only in $\mathbf{X}$, but in many situations we actually need to work with it in $\mathbf{C}$. Thus we need some ``concretization'' procedure. An obvious approach would be to require that $\gamma:\mathbf{C}\rightarrow \mathbf{X}$ reflects pullbacks, but this would assume that $\mathbf{C}$ has pullbacks, which is precisely what we are avoiding. One could also assume the presence of an opposite directed functor $F:\mathbf{X}\rightarrow \mathbf{C}$ which preserves pullbacks, but again this assumes that $\mathbf{C}$ has pullbacks. A final obvious attempt would be consider only flat functors $\gamma:\mathbf{C}\rightarrow \mathbf{X}$, but this is too restrictive for our purposes. The exact condition that we need is the following\footnote{One can actually work in a more general setting, as those discussed in \cite{eu_1}. However, the present conditions are sufficient for essentially everything, including our main results.}:

\begin{definition}
 Let $\gamma:\mathbf{C}\rightarrow \mathbf{X}$ be a faithful functor and $X\in \mathbf{X}$. A \textit{concreteness structure} for $\gamma$ in $X$ is given by a set $W_{\gamma}(A;X)$ of morphisms $\gamma(A)\rightarrow X$ for each $A\in \mathbf{C}$. Let 

 $$W_{\gamma}(X)=\bigcup_{A\in \mathbf{C}} W_{\gamma}(A;X).$$
 We say $\gamma$ satisfies the \textit{concreteness property} in the concreteness structure $W_{\gamma}(X)$ if for every pair $f:\gamma(A)\hookrightarrow X$ and $f':\gamma(B)\hookrightarrow X$ in $W_{\gamma}(X)$ there exist:
 \begin{enumerate}
     \item an object $A\cap_{X,\gamma}B\in\mathbf{C}$ and an isomorphism $u:\gamma(A\cap_{X,\gamma}B)\simeq \operatorname{pb}(f,f')$;
     \item morphisms $\theta_1:A\cap_{X,\gamma}B \rightarrow A$ and $\theta_2:A\cap_{X,\gamma}B \rightarrow B$,
 \end{enumerate}
such that the diagram below commutes. 
\begin{equation}\label{diagram_concreteness}
\xymatrix{\gamma (A\cap_{X,\gamma}B) \ar@/_{0.3cm}/[ddr]_{\gamma(\theta_1)} \ar[rd]^{u}_{\simeq} \ar@/^{0.3cm}/[rrd]^{\gamma(\theta_2)} \\
& \ar[d]_-{\pi_1} \operatorname{pb}(f,f') \ar[r]^-{\pi_2} & \gamma(B) \ar[d]^{f'} \\
& \gamma(A) \ar[r]_-{f}  & X}    
\end{equation}
\end{definition}
\begin{remark}
The cospan in $\mathbf{C}$ defined by $(\theta_1,\theta_2)$ is called \textit{concrete cospan} of the span $(f,f')$. Thus, the concreteness property allow us to replace a span in $\mathbf{X}$ by a not necessarily universal cospan (i.e., not necessarily a limit) in $\mathbf{C}$.
\end{remark}
\begin{lem}\label{lemma_monomorphisms}
If $\gamma$ satisfies the concreteness property in $W_{\gamma}(X)$ and if $f$ and $f'$ are monomorphisms, then the corresponding $\theta_1$ and $\theta_2$ are too.
\end{lem}
\begin{proof}
By the comutativity of the diagram above, we have $\gamma(\theta_i)=\pi_i\circ u$, with $i=1,2$. Since $f$ and $f'$ are monomorphisms, it follows that $\pi_i$ are monomorphisms too. Thus, $\gamma(\theta_i)$ also are. Since $\gamma$ is faithful, it reflect monomorphisms. 
\end{proof}
We can now put the previous discussion in our context.
\begin{definition}
 Let $\mathcal{X}=(\mathbf{X},\gamma,\gamma_\Sigma)$ be a $\Gamma$-ambient and let $X\in \mathbf{X}$. Define
 \begin{enumerate}
     \item  $W_{\gamma}(B;X)$ as the collection of all monomorphisms $\imath: \gamma(B)\hookrightarrow X$, where $B\in\mathbf{NFre}_{\Gamma}$, and $W_{\gamma}(X)$ as the union of all of them;
     \item $W_{\gamma_{\Sigma}}(B;X)$ as the collection of all morphisms $f:\gamma_{\Sigma}(B)\rightarrow X$, for $B\in \mathbf{NFre}_{\Gamma\times\Gamma}$ which are of the form $f=\imath\circ g$, where $\imath:\gamma_{\Sigma}(B')\hookrightarrow X$ is a monomorphism and $g:\gamma_{\Sigma}(B)\rightarrow \gamma_{\Sigma}(B')$ is an arbitrary morphism. Let $W_{\gamma_{\Sigma}}(X)$ be the union of all of them.
 \end{enumerate}
 We say that $\mathcal{X}$ is \textit{$\gamma$-concrete} (resp. \textit{$\gamma_{\Sigma}$-concrete}) in $X$ if the functor $\gamma$ (resp. $\gamma_\Sigma$) satisfies the concreteness property in $W_{\gamma}(X)$ (resp. $W_{\gamma_{\Sigma}}(X)$). We say that $\mathcal{X}$ is \textit{concrete} if it is both $\gamma$-concrete and $\gamma_{\Sigma}$-concrete.
\end{definition}
\begin{remark}\label{remark_need_concreteness}
The class $W_{\gamma_{\Sigma}}(B;X)$ contains morphisms of the form $f=\imath\circ g$ (instead of only monomorphism) precisely because we will need to concrectify diagrams like (\ref{diagram_abstract_intersection}), where $g=\gamma_\Sigma(*)$. Notice that if $g$ was assumed to be a monomorphism, then by the arguments of Lemma \ref{lemma_monomorphisms}, the multiplication $*:B\otimes B\rightarrow B$ would be a monomorphism too, which is not true in full generality.
\end{remark}
\begin{definition}
 Let $\mathbb{X}=(\mathcal{X},X,\imath,\imath')$ be an intersection structure between two fixed $\Gamma$-spaces $B$ and $B'$. We say that it is \textit{$\gamma$-concrete}\footnote{In \cite{eu_1} this corresponds to the notion of \textit{proper} intersection structure.} (resp. \textit{$\gamma_\Sigma$-concrete} or concrete) if $\Gamma$-ambient $\mathcal{X}$ is $\gamma$-concrete (resp. $\gamma_\Sigma$-concrete or concrete) in $X$.
\end{definition}
\begin{definition}
 Consider, now, an intersection structure $\mathbb{X}=(\mathcal{X},\mathbb{X}_0,\mathbb{X}_*,\mathbb{X}_+)$ between two distributive $\Gamma$-spaces $(B,\epsilon,\delta,*,+)$ and $(B',\epsilon',\delta',*',+')$. We say that $\mathbb{X}$ is \textit{concrete} if $\mathbb{X}_0$ is $\gamma$-concrete and $\mathbb{X}_*$ and $\mathbb{X}_+$ are $\gamma_{\Sigma}$-concrete. The \textit{concrete intersection} between $B$ and $B'$ in $\mathbb{X}$ is given by the concrete cospans of the spans\footnote{Take a look at diagrams (\ref{diagram_abstract_intersection}) and (\ref{diagram_concreteness}), and at Remark \ref{remark_need_concreteness}.} $(\imath,\imath')$, $(\jmath_*,\jmath_*')$, $(\jmath_+,\jmath_+')$, $(\jmath_*\circ \gamma_{\Sigma}(*),\jmath_*'\circ \gamma_{\Sigma}(*'))$ and $(\jmath_+\circ \gamma_{\Sigma}(+),\jmath_+'\circ \gamma_{\Sigma}(+'))$. The vertices of these cospans will be respectively denoted by $B\cap_{X,\gamma}B'$, $B_\epsilon\cap_{X_*,\gamma_{\Sigma}}B_{\epsilon'}$, $B_\delta\cap_{X_+,\gamma_{\Sigma}}B_{\delta'}$, $*\cap_{X_*,\gamma_{\Sigma}}*'$ and $+\cap_{X_+,\gamma_{\Sigma}}+'$.   
\end{definition}
\begin{remark}
Notice that if an abstract intersection between $\Gamma$-spaces $B$ and $B'$ is concrete, then it can itself be regarded as a $\Gamma$-space. However, if an abstract intersection between \textit{distributive} $\Gamma$-space is concrete, then it is not necessarily a \textit{distributive} $\Gamma$-space; it is only a $\Gamma\times \Gamma$-space with further stuff. This is enough for our purposes. 
\end{remark}

 We close with a useful lemma which shows that for concrete vector ambients, concrete intersections of nontrivial distributive $\Gamma$-spaces are nontrivial too.
\begin{lem}\label{vectorial_intersection}
Let $\mathbb{X}$ be a concrete intersection structure between distributive $\Gamma$-spaces $(B,\epsilon,\delta,*,+)$ and $(B',\epsilon',\delta',*',+')$, whose underlying $\Gamma$-ambient $\mathcal{X}$ is vectorial. If there exist $i,j\in \Gamma$ such that $B_i$ or $B'_j$ are nontrivial, then $\mathbb{X}$ is nontrivial too.
\end{lem}
\begin{proof}
See Proposition 1, page 10, of \cite{eu_1}.
\end{proof}

\subsection{Extension to Sheaves}
In the last sections we discussed that in order to formulate a notion of ``regularity'' on a $C^k$-manifold $M$ we have to consider nontrivial intersections between the existing $C^k$ regularity and the additional one. Notice, on the other hand, that by the very notion of regularity, they are \textit{local} properties. This means that a manifold $M$ is ``regular'' when the regularity is satisfied on the local pieces $U_s\subset M$ in such a way that coherence conditions hold at the intersections $U_{ss'}=U_s\cap U_{s'}$. Thus, we do not need a \textit{global} $\Gamma$-space $B(M)$, but actually one $B(\varphi_s(U_s))$ for each local chart $\varphi_s:U_s\rightarrow \mathbb{R}^n$. Therefore, we have to work with presheaves $B:\operatorname{Open}(\mathbb{R}^n)^{op}\rightarrow \mathbf{NFre}_{\Gamma}$ of $\Gamma$-spaces.

\begin{example}\label{example_Ck_presheaf}
The examples of $\Gamma$-spaces in Example \ref{example_1_gamma_spaces} and Example \ref{example_2_gamma_spaces} extends, naturally, to presheaves of $\Gamma$-spaces. E.g, given $m,k$ and $\beta:\Gamma \rightarrow [0,k]$, we have the presheaf $U\mapsto (C^{k-\beta(i)}(U))_i$. It will be denoted by $C^{k-\beta}_m$. or simply $C^{k-\beta}$ if $m=1$. If, in addition, $\Gamma = [0,k]$ and $\beta=id$ it will be denoted by $C^{k-}$.
\end{example}
\begin{example}
Let $B$ be a presheaf of $\Gamma$-spaces. Then every function $\alpha:\Gamma' \rightarrow \Gamma$ defines a new presheaf of $\Gamma'$-spaces $B_\alpha$ by $(B_\alpha)(U)=B(U)_\alpha$.
\end{example}

Notice that all the discussion above can be internalized in the presheaf category
of presheaves in $\mathbb{R}^{n}$ by means of just taking a parametrization of definitions and results in terms of opens sets of $\mathbb{R}^n$. Thus, we can talk about presehaves of distributive $\Gamma$-spaces, presheaves of intersection structures between presheaves of distributive $\Gamma$-spaces, and so on. We refer the reader to Section 3 of \cite{eu_1} for more details to this extension to the presheaf setting.

\begin{remark}
To maintain compatibility with notations of \cite{eu_1}, throughout this paper a presheaf of intersection structures between presheaves of distributive $\Gamma$-spaces $B$ and $B'$ will be called an \textit{intersection structure presheaf} (ISP). Furthermore, it will be denoted simply by $\mathbb{X}=(\mathcal{X},X)$, where $\mathcal{X}=(\mathbf{X},\gamma,\gamma_{\Sigma})$ is the $\Gamma$-ambient and $X:\operatorname{Op}(\mathbb{R}^n)^{op}\rightarrow \mathbf{NFre}_{\Gamma}$ is the presheaf of ambient $\Gamma$-spaces.
\end{remark}

\subsection{$B^k_{\alpha,\beta}$-Presheaves}

In order to prove the existence of a geometric object on a smooth manifold one typically first proves local existence and then uses a partition of unity to globalize the construction. To extend the use of partition functions to the present regular setting we will work with presheaves of $\Gamma$-spaces which are \textit{$B^k_{\alpha,\beta}$-presheaves}. See Section 3 of \cite{eu_1} for more details.

 \begin{itemize}
     \item In the following: $C^{k}_b(U)\subset C^{k}(M)$ denotes the space of $C^k$-bump functions with the Fr\'echet structure induced by the seminorms (\ref{frechet_ck}).
 \end{itemize}
\begin{definition}
We say that a set of indexes $\Gamma '$ has \textit{degree $r$} for some $0\leq r \leq k$ if it contains the interval $[0,r]$. 
\end{definition}
\begin{itemize}
    \item In the following we will write $\Gamma_r$ to denote a generic set with degree $r$, while $\Gamma$ will remain an arbitrary set.
\end{itemize}
\begin{definition}
Let $B=(B_i)_{i\in \Gamma}$ be a presheaf of $\Gamma$-spaces and let $\alpha:\Gamma' \rightarrow \Gamma$ and $\beta:\Gamma' \rightarrow [0,k]$ be functions. We say that the tuple $(B,\alpha,\beta,k)$ has \textit{degree $r$} if the domain $\Gamma'$ of the functions $\alpha$ and $\beta$ has degree $r$, i.e., if they are defined in some $\Gamma_r$.
\end{definition}

 \begin{definition}\label{definition_C^k-presheaves}
Let $(B,\alpha,\beta,k)$ be a tuple as above. Let $\mathbb{X}=(\mathcal{X},X)$ be a concrete ISP between $B_{\alpha}$ and $C^{k-\beta}$. We say that $(B,k,\alpha,\beta)$ is a $B^{k}_{\alpha,\beta}$\textit{-presheaf} in $\mathbb{X}$ if it has degree $k$ and for every open set $U\subset\mathbb{R}^n$ and every $i\Gamma_k$ there exists the dotted arrow $\star_{U,i}$ making commutative
the diagram below, where $\theta_{2,U}$ is part of the concrete cospan\footnote{which therefore is a monomophisms due to Lemma \ref{lemma_monomorphisms}.}.  
\begin{equation}{\label{diagram_star}
\xymatrix{C^{k}(U)\otimes C^{k-\beta(i)}(U) \ar[r]^-{\cdot_{U,i}} & C^{k-\beta(i)}(U) \\
\ar@{^(->}[u]^{\imath_U \otimes \theta_{2,U}} C^k_b(U) \otimes B_{\alpha (i)}(U) \cap_{X(U)} C^{k-\beta(i)}(U) \ar@{-->}[r]_-{\star_{U,i} } & B_{\alpha(i)}(U) \cap_{X(U)} C^{k-\beta(i)}(U) \ar@{^(->}[u]_{\theta_{2,U}} }}
\end{equation}
\end{definition}

Intuitively, a $B^k_{\alpha,\beta}$-presheaf is a presheaf of $\Gamma$-spaces $B(U)$, which, when regarded as a presheaf $B_{\alpha}$ of $\Gamma_k$-spaces via $\alpha$, has an intersection space with $C^{k-\beta}(U)$ that is closed under multiplication of $C^k$-bump functions.

\begin{example}\label{example_b_k_a_b_1}
For every $k,\beta,\alpha$ the presheaf of $C^{k-\alpha}$ is a
$B_{\beta\beta}^{k}$-presheaf in the ISP which is objectwise
the standard vectorial intersection structure of Example \ref{vectorial_ambient}.

\begin{example}\label{example_b_k_a_b_2}
Similarly, the presheaves $L_\alpha(U)_{i}=L^{\alpha(i)}(U)$ with the distributive structure given by pointwise sum and multiplication,
is a nice $B_{\beta}^{k,\alpha}$-presheaf in
the standard ISP, for $A(U)=C_{b}^{\infty}(U)$. An analogous
conclusion is valid if we replace pointwise multiplication with convolution
product.
\end{example}
\end{example}
We close this section with two remarks concerning Definition \ref{definition_C^k-presheaves}.
 \begin{enumerate}
     \item \label{nice_presheaves} It has a generalization where one requires invariance under multiplication only by a subvector space $A(U)\subset C^k(U)$ with a nuclear Fr\'echet structure. In this case, for instance, in Example \ref{example_b_k_a_b_2} we would get a different $B^k_{\alpha,\beta}$-structure in $L_\alpha$ taking $A(U)=\mathcal{S}(U)$ as the Schwartz space. Following \cite{eu_1}, a \textit{nice} $B^k_{\alpha,\beta}$-presheaf should be one such that $A(U)$ nontrivially intersects $C^k_b(U)$. In the following, however, we will work only with $A(U)=C^k_b(U)$, so that every $B^k_{\alpha,\beta}$-preshef will be nice.
     \item Following the conventions of \cite{eu_1} a $B^k_{\alpha,\beta}$-presheaf in $\mathbb{X}$ whose presheaves of intersection spaces $B_\alpha \cap_X C^{k-\beta}$ are nontrivial is called a \textit{$C^k_{\alpha,\beta}$-presheaf}. In this paper we will not work directly with them. However, from Lemma \ref{vectorial_intersection}, for vectorial ISP $\mathbb{X}$, if  $B_{\alpha}$ is nontrivial (i.e., if for every $U\subset \mathbb{R}^n$ there is $i_U\in \Gamma$ such that $B_{\alpha(i_U)}(U)\neq 0$), then every $B^k_{\alpha,\beta}$-presheaf in $\mathbb{X}$ is a $C^k_{\alpha,\beta}$-presheaf. In particular, the $B^k_{\alpha,\beta}$-presheaves of Examples \ref{example_b_k_a_b_1}-\ref{condition_2} are $C^k_{\alpha,\beta}$-presheaves.
\end{enumerate}

\subsection{$B^k_{\alpha,\beta}$-Functions}

As discussed in the introduction, in order to formalize the notion of regularity on manifolds we have to demand conditions not only on the transition functions $\varphi^{a}_{s's}:\varphi_{s}(U_{ss'})\rightarrow \mathbb{R}$, with $a=1,...,n$, but also on their derivatives $\partial^\mu\varphi^{a}_{s's}$, with $\mu=1,...,k$, where $\varphi_{s's}=\varphi_{s'}\circ \varphi^{-1}_s$.
In our context this is described as follows. See Section 5 of \cite{eu_1} for a complete exposition.

\begin{definition}
Let $(B,k,\alpha,\beta)$ be a $B^k_{\alpha,\beta}$-presheaf in a concrete ISP $\mathbb{X}$. Let $U\subset\mathbb{R}^n$ be an open set and let $f:U\rightarrow \mathbb{R}$ be a real $C^k$-function. Let $S\subset [0,k]\subset \Gamma_k$ be a subset\footnote{Here is where we are using that a $B^k_{\alpha,\beta}$-presheaf has degree $k$.}. We say that $f$ is a $(B^k_{\alpha,\beta}\vert S)$\textit{-function} (or $(B,k,\alpha,\beta|S)$\textit{-function}) in $\mathbb{X}$ if for every $i\in S$ we have $\partial^{i}f \in B_{\alpha(i)}(U)\cap_{X(U)}C^{k-\beta(i)}(U)$. A vectorial function $f:U\rightarrow\mathbb{R}^m$, with $m\geq 1$ is a $(B^k_{\alpha,\beta}|S)$\textit{-function} in $\mathbb{X}$ if its coordinate functions $f^a:U\rightarrow \mathbb{R}$, with $a=1,...,m$, are.
\end{definition}

Thus, for instance, if $\Gamma = [0,k]=\Gamma_k$ and $\alpha(i)=i=\beta(i)$, then $f$ is a $(B,k,\alpha,\beta|S)$-function in $\mathbb{X}$ precisely if $\partial^\mu f$ belongs to $B_i(U) \cap_X(U) C^{k-i}(U)$. This means that the set $S$ and the functions $\alpha$ and $\beta$ determine how the derivatives $\partial^\mu f$ of $f$ lose regularity when $\mu$ increases.

\subsection{$B^k_{\alpha,\beta}$-Manifolds}
We are now ready to define what is a $B^k_{\alpha,\beta}$-manifold. See Section 5 of \cite{eu_1}.

\begin{definition}
A $C^{k}$-\emph{manifold} is a paracompact Hausdorff topological
space $M$ endowed with a maximal atlas $\mathcal{A}$ with charts
$\varphi_{s}:U_{s}\rightarrow\mathbb{R}^{n}$, whose transition functions
$\varphi_{s's}=\varphi_{s'}\circ\varphi_{s}^{-1}:\varphi_{s}(U_{ss'})\rightarrow\mathbb{R}^{n}$
are $C^{k}$, where $U_{ss'}=U_{s}\cap U_{s'}$.
\end{definition}
\begin{definition}
Let $(M,\mathcal{A})$ be a $C^k$-manifold. A subatlas $\mathcal{A}'\subset \mathcal{A}$ is \textit{closed under restrictions} if restrictions of charts in $\mathcal{A'}$ to smaller open sets remains in $\mathcal{A'}$. More precisely, if $(\varphi,U)\in \mathcal{A}'$ and $V\subset U$, then $(\varphi\vert_V,V)\in \mathcal{A}'$. 
\end{definition}
\begin{definition}
Let $(B,k,\alpha,\beta)$ be a $B^k_{\alpha,\beta}$-presheaf in a concrete ISP $\mathbb{X}$. A $(B_{\alpha,\beta}^{k},\mathbb{X})$\emph{-structure
}in a $C^{k}$-manifold $(M,\mathcal{A})$ is a subatlas $\mathcal{B}_{\alpha,\beta}^{k}(\mathbb{X})$ closed under restrictions
whose transition functions $\varphi_{s's}$ are $(B,k,\alpha,\beta)$-functions in $\mathbb{X}$. Explicitly, this means that if $\varphi_s\in \mathcal{B}_{\alpha,\beta}^{k}(\mathbb{X})$, then for each other $\varphi_{s'}\in \mathcal{B}_{\alpha,\beta}^{k}(\mathbb{X})$ such that $U_{ss'}\neq \varnothing$, we have 
$$
\partial^{i}\varphi^a_{ss'}\in B_{\alpha(i)}(\varphi_s(U_{ss'}))\cap_{X(\varphi_s(U_{ss'}))}C^{k-\beta(i)}(\varphi_s(U_{ss'}))
$$
for every $i\in [0,k]\subset \Gamma_k$.
\end{definition}
\begin{definition}
A \emph{$(B_{\alpha,\beta}^{k},\mathbb{X})$-manifold
}is a $C^{k}$ manifold endowed with a $(B_{\alpha,\beta}^{k},\mathbb{X})$-structure.
\end{definition}

In \cite{eu_1} the authors define \textit{morphisms} between two $(B^k_{\alpha,\beta},\mathbb{X})$-manifolds as $C^k$-maps whose local representation by charts in the corresponding $(B^k_{\alpha,\beta},\mathbb{X})$-structures are $(B,k,\alpha,\beta)$-functions in $\mathbb{X}$. Here we will need only the particular case between $C^k$-manifolds.
\begin{definition}
Let $(M,\mathcal{A})$ and $(M',\mathcal{A}')$ be $C^k$-manifolds, let $\mathbb{X}$ a concrete ISP, $(B,k,\alpha,\beta)$ a $B^k_{\alpha,\beta}$-presheaf in $\mathbb{X}$ and $S\subset [0,k]\subset \Gamma_k$ a subset. A \textit{$(B,k,\alpha,\beta|S)$-morphism}, \textit{$(B,k,\alpha,\beta|S)$-function} or \textit{$(B^k_{\alpha,\beta}|S)$-morphism}  between $M$ and $M'$ is a $C^k$-function $f:M\rightarrow M'$ such that for every $\varphi\in\mathcal{A}$ and $\varphi ' \in \mathcal{A'}$ the corresponding $C^k$-function $\varphi' \circ f \circ \varphi^{-1}$ is actually a $(B^k_{\alpha,\beta}|S)$-function in $\mathbb{X}$. If $S=[0,k]$ we say simply that $f$ is a \textit{$B^k_{\alpha,\beta}$-morphism}.
\end{definition}

\begin{remark}
From the examples above it is clear that a $C^k$-manifold does not necessarily admits any $(B^k_{\alpha,\beta},\mathbb{X})$-structure. Thus, there are obstructions which typically can be characterized as topological or geometric obstructions on the underlying $C^k$-structure. A complete characterization of these obstructions was not found yet. On the other hand, in \cite{eu_1} the authors give sufficient conditions on $B$, $\alpha$, $\beta$, $k$ and on the underlying $C^k$-structure ensuring the existence of $B^k_{\alpha,\beta}$-structures, specially in the case where $(B,k,\alpha,\beta)$ is a $C^k_{\alpha,\beta}$-presheaf.
\end{remark}

\subsection{Affine $B^k_{\alpha,\beta}$-Connections}\label{sec_affine_connections}

We now introduce the objects of study in this paper: affine $\nabla$ connections on $C^k$-manifolds $M$ whose local coefficients $\Gamma^c_{ab}$ satisfies additional regularity conditions, possibly depending on some regularity on the underlying manifold $M$. In other words, we will consider ``regular'' affine connections on $(B^k_{\alpha,\beta},\mathbb{X})$-manifolds. The obvious idea would then to be take affine connections whose coefficients $(\Gamma_\varphi)^c_{ab}:\varphi(U)\rightarrow \mathbb{R}$ in each $(\varphi,U)\in \mathcal{B}^k_{\alpha,\beta}(\mathbb{X})$ are $B^k_{\alpha,\beta}$-functions. This, however, is too restrictive to the global existence of such objects due to two reasons.
\begin{enumerate}
    \item  \label{condition_1} \textit{The regularity rate decay of the derivatives $\partial^\mu \Gamma^c_{ab}$, when $\mu$ grows, is typically different of that $\partial^{\mu}\varphi^a_{ji}$ of the transition functions $\varphi^a_{ji}$}. This leads us instead to consider affine connections whose coefficients are $(B^k_{\alpha',\beta'}|S)$-functions, where $\alpha',\beta'$ are functions typically \textit{different} from $\alpha,\beta$ in shape and/or in domain/codomain. For instance, $\varphi^a_{ji}$ are $C^k$, while $\Gamma^c_{ab}$ is $C^{k-2}$. Thus, if $\beta(i)=i$, then $\beta'$ must be $\beta'(i)=i+2$. Equivalently, we can regard $\beta'$ as defined in $[2,k]$ instead of in $[0,k]$ and take $\beta'(i)=i$. 
    \item \label{condition_2} \textit{The ISP appearing in the regularity of $\Gamma^c_{ab}$ is usually different from that describing the regularity of $\varphi_{ij}$}. For instance, we could consider intersections in the same $\Gamma$-ambient $(\mathbf{X},\gamma,\gamma_{\Sigma})$, but in different presheaves of ambient $\Gamma$-spaces $X$ and $Y$. 
\end{enumerate}
\begin{definition}
Let $(B,k,\alpha,\beta)$ be a $B^k_{\alpha,\beta}$-presheaf in a concrete ISP $\mathbb{X}$. Let $(M,\mathcal{B}^k_{\alpha,\beta}(\mathbb{X}))$ be a $B^k_{\alpha,\beta}$-manifold in $\mathbb{X}$. Let $S \subset [0,k]\subset \Gamma_k$ be a subset and consider functions $\alpha':\Gamma_k \rightarrow \Gamma$ and $\beta':\Gamma_k \rightarrow [2,k]$. Let $\mathbb{Y}$ be a concrete ISP between $B_{\alpha'}$ and $C^{k-\beta'}$. An \textit{affine $(B^k_{\alpha',\beta'}|S)$-connection} in $M$, relative to $\mathbb{Y}$, is an affine connections $\nabla$ in the underlying $C^k$-manifold such that for every $(\varphi,U) \in \mathcal{B}^{k}_{\alpha,\beta}(\mathbb{X})$ the corresponding local coefficients $(\Gamma_\varphi)^c_{ab}:\varphi(U)\rightarrow \mathbb{R}$ are $(B,k,\alpha',\beta'|S)$-functions in $\mathbb{Y}$.
\end{definition}
\subsection{Regularity Globalization Lemma}\label{sec_globalization_lemma}

 Condition \ref{condition_1} and the nice hypotheses on $B^k_{\alpha,\beta}$ suffice to ensure existence of connections $\nabla$ which are \textit{locally regular} in $\mathbb{X}$, following the classical construction of affine connections via partitions of unity. Condition \ref{condition_2} will be used to prove global regularity. More precisely, we show that if $B_{\alpha}$ admits some universal way to connect two ISP from it to $C^{k-\beta}$, then every affine connection $\nabla$ \textit{locally regular} in a given $\mathbb{X}$ induces another affine connection $\overline{\nabla}$ which is \textit{globally regular} in an ISP $\mathbb{Y}$ ``compatible'' with $\mathbb{X}$. Here we introduce these notions of universality and compatibility between two ISP.
\begin{definition}$\;$

\begin{enumerate}
    \item Two ISP $\mathbb{X}$ and $\mathbb{Y}$ between presheaves of $\Gamma$-spaces $B$ and $B'$ are \textit{compatible} if they have the same underlying $\Gamma$-ambient and \textit{strongly compatible} if they also have the same presheaf of ambient spaces $X$;
    \item A \textit{strongly compatible sequence of} ISP (or scISP, for short) between sequences $B=(B_i)$ and $B=(B'_i)$ of presheaves of $\Gamma$-spaces is a sequence $\underline{\mathbb{X}}=(\mathbb{X}_i)$ such that $\mathbb{X}_i$ and $\mathbb{X}_j$ are strongly compatible for every $i,j$; 
    \item Two sequences $\underline{\mathbb{X}}=(\mathbb{X}_i)$ and $\underline{\mathbb{Y}}=(\mathbb{Y}_i)$ of ISP between $B=(B_i)$ and $B'=(B'_{i})$ are \textit{compatible} if $\mathbb{X}_i$ and $\mathbb{Y}_i$ are compatible for every $i$.
\end{enumerate}   
\end{definition}
\begin{definition}
Let $B$ and $B'$ be two presheaves of $\Gamma$-spaces and of $\Gamma'$-spaces, respectively. Let $\Gamma_k$ be a set of degree $k$. A \textit{$\Gamma_k$-connective structure} between $B$ and $B'$ is given by:
\begin{enumerate}
    \item sets of functions $\mathcal{O}\subset \operatorname{Mor}(\Gamma_k,\Gamma)$ and $\mathcal{O}\subset \operatorname{Mor}(\Gamma_k,\Gamma')$;
    \item functions $D_{\mathcal{O}}$ and $D_{\mathcal{Q}}$ assigning to each pairs  $\theta,\theta'\in\mathcal{O}$ and $\vartheta,\vartheta'\in\mathcal{Q}$ corresponding morphisms $D_\mathcal{O}(\theta,\theta'):B_{\theta}\Rightarrow B_{\theta'}$ and $D_\mathcal{Q}(\vartheta,\vartheta'): B'_{\vartheta}\Rightarrow B'_{\vartheta'}$ which satisfy the composite laws $$D_\mathcal{O}(\theta,\theta'')=D_\mathcal{O}(\theta',\theta'')\circ D_\mathcal{O}(\theta,\theta'),$$ meaning that the following triangles are commutative:
    \begin{equation}\label{composition_law_connections}
    \xymatrix{
        B_\theta \ar@{=>}[rrd]_{D_{\mathcal{O}}(\theta,\theta'')} \ar@{=>}[rr]^{D_{\mathcal{O}}(\theta,\theta')} && B_{\theta'} \ar@{=>}[d]^{D_{\mathcal{O}}(\theta',\theta'')} &&& \ar@{=>}[d]_{D_{\mathcal{Q}}(\vartheta',\vartheta'')} B_{\vartheta'}&& B_{\vartheta} \ar@{=>}[ll]_{D_{\mathcal{Q}}(\vartheta,\vartheta') } \ar@{=>}[lld]^{D_{\mathcal{Q}}(\vartheta,\vartheta'') } \\
        && B_{\theta''} &&& B_{\vartheta ''} }
    \end{equation}
    \item  for every $(\theta,\vartheta)\in\mathcal{O}\times \mathcal{Q}$ a concrete ISP $\mathbb{X}_{\theta,\vartheta}$ between $B_\theta$ and $B'_{\vartheta}$ whose corresponding sequence $\underline{\mathbb{X}}=(\mathbb{X}_{\theta,\vartheta})$ is strongly compatible. Let $X$ denote the underlying presheaf of ambient spaces;
    \item  for every other scISP $\underline{\mathbb{Y}}=(\mathbb{Y}_{\theta,\vartheta})$ between $B$ and $B'$, indexed in $\mathcal{O}\times \mathcal{Q}$ and compatible with $\underline{\mathbb{X}}=(\mathbb{X}_{\theta,\vartheta})$, a morphism $\mathcal{D}(X;Y):X\Rightarrow Y$ of presheaves, where $Y$ is the presheaf of ambient spaces of each $\mathbb{Y}_i$, such that the diagram below commutes for every $\theta,\theta',\vartheta,\vartheta'$. 
    \begin{equation}{\label{connection_on_B}
\xymatrix{ \ar@{==>}[rd]^{\xi^{\theta,\vartheta}_{\theta'\vartheta'}} \ar@{=>}[dd] (B_{\theta} \cap _{X,\gamma} B'_{\vartheta}) \ar@{=>}[rr] &&  B'_{\vartheta} \ar@{=>}@/^{0.75cm}/[ddd]^{\jmath'_{\vartheta}} \ar@{=>}[d]_{D_{\mathcal{Q
}}(\vartheta ,\vartheta ')} \\
& \ar@{=>}[r] (B_{\theta '} \cap_{Y} B'_{\vartheta '}) \ar@{=>}[d] & B'_{\vartheta '} \ar@{=>}[d]^{\imath_{\vartheta'} '} \\
 B_{\theta} \ar@{=>}[r]^{D_{\mathcal{O}}(\theta ,\theta ')} \ar@{=>}@/_{0.5cm}/[rrd]_{\jmath_{\theta}} & \ar@{=>}[r]_{\imath_{\theta'}} B_{\theta '} & Y 
\\
&& X\ar@{=>}[u]^{\mathcal{D}(X;Y)}}}
\end{equation}
\end{enumerate}
\end{definition}
\begin{remark}\label{remark_composition_law_connections}
By the universality of pullbacks, there exists the dotted arrow $\xi^{\theta,\vartheta}_{\theta',\vartheta'}$ in diagram (\ref{connection_on_B}). Furthermore, the composition laws (\ref{composition_law_connections}) and the uniqueness of pullbacks imply a composition law for those dotted arrows: 
$$
\xi^{\theta,\vartheta}_{\theta '',\vartheta ''}=\xi^{\theta',\vartheta'}_{\theta '',\vartheta ''}\circ \xi^{\theta,\vartheta}_{\theta ',\vartheta '}.
$$
\end{remark}
\begin{definition}
We say that the sequence $\underline{X}=(\mathbb{X}_{\theta,\vartheta})$ is the \emph{base} scISP of the $\Gamma_k$-connective structure. 
\end{definition}
\begin{itemize}
    \item In the following we will say ``consider a $\Gamma_k$-connective structure between $B$ and $B'$ in the scISP $\underline{\mathbb{X}}$'', meaning that  $\underline{\mathbb{X}}$ is the base scISP of the refereed connective structure.
\end{itemize}

In the case where $B$ and $B'$ are presheaves of distributive $\Gamma$-spaces and $\Gamma'$-spaces, we need to require that the connection between them preserves sum and multiplication. 
 \begin{definition}\label{distributive_connection}
 Let $B$ and $B'$ be presheaves of distributive $\Gamma$-spaces and $\Gamma'$-spaces, respectively. A  $\Gamma_k$-connective structure between $B$ and $B'$ in $\underline{\mathbb{X}}=(\mathbb{X}_{\theta,\vartheta})$ is \textit{distributive} if for every $\theta,\theta'\in\mathcal{O}$ and every $\vartheta,\vartheta'\in\mathcal{Q}$ there are objective monomorphisms $(\gamma_*)_{\theta,\theta'}:B_{\epsilon(\theta,\theta)}\Rightarrow B_{\epsilon(\theta',\theta')}$
and $(\gamma'_{*'})_{\vartheta,\vartheta'}:B'_{\epsilon'(\vartheta,\vartheta)}\Rightarrow B'_{\epsilon'(\vartheta',\vartheta')}$
making commutative the diagram below, and also objectwise monomorphisms
$(\gamma_{+})_{\theta,\theta'}:B_{\delta(\theta,\theta)}\Rightarrow B_{\delta(\theta',\theta')}$
and $(\gamma'_{+'})_{\vartheta,\vartheta'}:B'_{\delta'(\vartheta,\vartheta)}\Rightarrow B'_{\delta'(\vartheta',\vartheta')}$
making commutative the analogous diagram for the additive structures.
$$
\resizebox{\displaywidth}{!}{
\xymatrix{\ar@{=>}[d]_{(\gamma' _{*'})_{\vartheta,\vartheta'}} B'_{\epsilon '(\vartheta,\vartheta)} & \ar@{=>}[l]_-{*'}   B'_{\vartheta} \otimes B'_{\vartheta} & \ar@{=>}[d]_{\xi_{\theta ',\vartheta '}^{\theta, \vartheta}\otimes \xi_{\theta ',\vartheta '}^{\theta ,\vartheta}} (B_{\theta} \cap_X B'_{\vartheta})\otimes (B_{\theta} \cap_X B'_{\vartheta}) \ar@{=>}[r] \ar@{=>}[l] & B_{\theta} \otimes B_{\theta} \ar@{=>}[r]^-{*} & B_{\epsilon (\theta,\theta)} \ar@{=>}[d]^{(\gamma _{*})_{\theta,\theta'}}   \\
B'_{\epsilon '(\vartheta ',\vartheta ')} & \ar@{=>}[l]^-{*'} B'_{\vartheta '} \otimes B'_{\vartheta '} & (B_{\theta'} \cap_{Y} B'_{\vartheta '})\otimes (B_{\theta '} \cap_Y B'_{\vartheta '}) \ar@{=>}[r] \ar@{=>}[l] & B_{\theta '} \otimes B_{\theta '} \ar@{=>}[r]_-{*} & B_{\epsilon (\theta ',\theta ')}}}
$$

\end{definition}
Let us now return to our context of $B^k_{\alpha,\beta}$-presheaves. First, a notation:
\begin{itemize}
    \item given a set $\Gamma_k$ of degree $k$, a function $\beta_0:\Gamma_k \rightarrow [2,k]$ and $j\in \Gamma_k$, let $[\beta_0;j]_k$ denote the subset of all $i\in [0,k]\subset \Gamma_k$ such that $0\leq \beta_0(j)-i \leq k$, i.e., the interval $[0,k-\beta_0(j)]$.
    \item Thus, by restriction each presheaf $B$ of $\Gamma'$-spaces and the presheaf\footnote{Recall the notation in Example \ref{example_Ck_presheaf}.} $C^{k-}$ of $[0,k]$ can both be regarded as presheaves of   $[\beta_0;j]_k$-spaces.
 \end{itemize}

\begin{definition}
Let $\mathbb{X}$ be a concrete ISP and let $(B,k,\alpha,\beta)$ be a $B^k_{\alpha,\beta}$-presheaf in $\mathbb{X}$. Given functions $\alpha_0:\Gamma_k\rightarrow \Gamma$ and $\beta_0:\Gamma_k \rightarrow [2,k]$, for each $j\in\Gamma_k$ regard $B$ and $(C^{k-i})_i$ as presheaves of $[\beta_0;j]_k$-spaces. A $(\alpha_0,\beta_0;j)$\textit{-connection} in $B$ is a connection between $B$ and $C{^k-}$ in some scISP $\underline{\mathbb{X}}$, such that:
\begin{enumerate}
    \item $\alpha,\alpha_0\in \mathcal{O}$, $\beta,\beta_0\in \mathcal{Q}$ and $\mathbb{X}_{\alpha,\beta}=\mathbb{X}$;
    \item  $\alpha_{0,j}\in\mathcal{O}$ and $\beta_{\#,j}\in\mathcal{Q}$, where $\alpha_{0,j}(i)=\alpha_0(j)$ is the constant function and $\beta_{,j}(i)=\beta_0(j)-i$ is the shifting function;
    \item if $\vartheta \mathcal{Q}$ is any function bounded from above by $\beta_0$, i.e., if $\vartheta(i)\leq \beta_0(i)$, then the morphism $D_{\mathcal{Q}}(\beta_0,\vartheta):C^{h-\beta_0}\Rightarrow C^{k-\vartheta}$ is the canonical inclusion.
\end{enumerate}
\end{definition}
\begin{definition}
A $(\alpha_0,\beta_0,j)$-connection in $B$ is \textit{distributive} if it is a distributive connection between $B$ and $C^{k-}$ in the sense of Definition \ref{distributive_connection}.
\end{definition}
The globalization of the local regularity of an affine connection in a $(B^k_{\alpha,\beta},\mathbb{X})$-manifold will depends on the existence of a distributive $(\alpha_0,\beta_0;j)$-connection in $B$. Furthermore, the extended regular will have local coefficients which are $(B,\theta,\vartheta,\mathbb{Y}|S)$-functions for functions $(\theta,\vartheta)\in \mathcal{O}\times \mathcal{Q}$ which are ``ordinary'' and for a set $S$ which is of a very special shape, in the the following sense.

\begin{definition}
An \textit{additive structure} in a set $\Gamma_k$ of degree $k$ consists of 
\begin{enumerate}
    \item a set $\Gamma_{2k}$ of degree $2k$ containing $\Gamma_k$;
    \item a map $\overline{+}:\Gamma_k\times \Gamma_k \rightarrow \Gamma_{2k}$ extending sum of nonnegative integers, i.e., such that if $z,l\in [0,k]$, then $z\overline{+}l=z+l$. 
\end{enumerate}
An \textit{additive set of degree $k$} is a set of degree $k$ where an additive structure where fixed.  
\end{definition}
\begin{itemize}
    \item To make notations simpler, in the following we will write $z+l$ instead  of $z\overline{+}l$ even if $z,l$ are arbitrary elements of $\Gamma_k$.
    \item Let $(\Gamma_k,+)$ be an additive set of degree $k$ and $\beta_0:\Gamma_k \rightarrow [2,k]$ be a map. Given $z\in [\beta_0;j]_k$, let $\Gamma_k[z] \subset \Gamma_k$ be the set of every $l\in \Gamma_k$ such that $z+l\in [\beta_0;j]_k$.
\end{itemize}

\begin{lem}\label{lema_S_set}
In the same notations above, for every $z\in [\beta_0;j]_k$ we have  
\begin{equation}\label{dependence_S_z}
    \Gamma_k[z]\subset (\Gamma_k - \Gamma_k\cap[\beta_0(j)-z+1,\infty)).
\end{equation}

\end{lem}
\begin{proof}
Straightforward: just a simple inequalities study.
\end{proof}
\begin{example}\label{example_S_set}
In particular, if $\Gamma_k=[0,k]$, then $\Gamma_k[z]\subset [0,\beta_0(j)-z]$. More specifically, if $z=0$, then $\Gamma_k[z]=[\beta_0;j]_k$, while if $z=k-\beta_0(j)$, then $\Gamma_k[z]=0$.
\end{example}
\begin{remark}[Important Remark]\label{important_remark}
Sets like $\Gamma_k[z]$ will be the sets $S$ in which the regular affine connections that we will build will have coefficients as $(B,k,\theta,\vartheta|S)$-functions. Thus, for larger $\Gamma_k[z]$ we will have more control on the regularities of higher order derivatives of the coefficient functions. On the other hand, from Lemma \ref{lema_S_set} we see that when $z$ increases, the set $\Gamma_k[z]$ becomes smaller. Furthermore, in Example \ref{example_S_set} we saw that if $\Gamma_k=[0,k]$ is the obvious set of degree $k$, then $\Gamma_k[z]$ strongly depends on $z$. In particular, for $z=k-\beta_0(j)$ we have $\Gamma_k[z]=0$ which would produce a regular affine connection with no control on the regularity of derivatives! This is one of the main reasons for working in the setup of general $\Gamma_k$-spaces: \textit{to have more control on the regularity of derivatives of the new affine connections, specially independing on $z$}. Indeed, notice that if $\Gamma_k$ increases, then  the difference (\ref{example_S_set}) becomes closer to $\Gamma_k$ and, therefore, is independent of $z$.
\end{remark}
\begin{example}\label{example_smooth_case}
If $k=\infty$, then $[\beta_0;j]_\infty=[0,\infty)$ for every $\beta_0$. Thus $[0,\infty)\subset \Gamma_{\infty}[z]$ for every $z\in [\beta_0;j]_\infty=[0,\infty)$, which follows from the fact that the additive structure of $\Gamma_\infty$ is compatible with sum of integers.
\end{example}

\begin{definition}
Let $\Gamma_k$ be a set of degree $k$, $\Gamma$ a set and let   $\mathcal{O}\subset\operatorname{Mor}(\Gamma_k;\Gamma)$ and $\mathcal{Q}\subset\operatorname{Mor}(\Gamma_k;[0,k])$. Let $X\subset \Gamma_r$ a subset. We say that a sequence of pairs  $(\theta_{l},\vartheta_{l})_l$, with $\theta_l\in\mathcal{O}$ and $\vartheta_l\in \mathcal{Q}$, is \emph{ordinary in $X$} if each $l\in X$ and the corresponding pair $(\theta_{*},\vartheta_{*})$,
given by $\theta_{*}(l)=\theta_{l}(z)$ and $\vartheta_{*}(l)=\vartheta_{l}(z)$ also
belongs to $\mathcal{O}\times\mathcal{Q}$. Finally, we say that $(\theta,\vartheta)\in\mathcal{O}\times\mathcal{Q}$
is \emph{ordinary in $X$} if $\theta=\theta'_{*}$ and $\vartheta=\vartheta'_{*}$
for some $\theta'_{*}$ and $\vartheta'_{*}$, i.e, if it is the induced
pair of a sequence of ordinary pairs in $X$.
\end{definition}

We can now state and prove the fundamental step in the globalization of the regularity.

\begin{lem}[Regularity Globalization Lemma]
\label{key_lemma} Let $M$ be a $(B_{\alpha,\beta}^{k},\mathbb{X})$-manifold and suppose that $B$ admits a $(\alpha_{0},\beta_{0};j)$-connection in some \emph{scISP} $\underline{\mathbb{X}}$. Then, for every $U\subset\mathbb{R}^{n}$, every
$z\in [\beta_0;j]_k$ and every \emph{scISP}  $\underline{\mathbb{Y}}=(\mathbb{Y}_{\theta,\vartheta})$ between $B$ and $C^{k-}$, compatible with $\underline{\mathbb{X}}$, the elements of $$B_{\alpha_{0}(j)}(U)\cap_{X(U)}C^{k-\beta_{0}(j)-z}(U)$$
can be regarded as $(B,k,\theta,\vartheta|\Gamma_k[z])$-functions in $\mathbb{Y}_{\theta,\vartheta}$, for every $(\theta,\vartheta)\in\mathcal{O}\times\mathcal{Q}$ ordinary in $\Gamma_k[z]$.
\end{lem}
\begin{proof}
Since the $B$ and $C^{k-}$ admits a $\Gamma_k$-connective structure in some $\underline{\mathbb{X}}$,
by definition it follows that for every scISP $\underline{\mathbb{Y}}$ compatible
with  $\underline{\mathbb{X}}$ and every $\theta,\theta'\in\mathcal{O}$
and $\vartheta,\vartheta'\in\mathcal{Q}$ we have morphisms $\xi_{U,i}:B_{\theta(i)}(U)\cap_{X(U)}C^{k-\vartheta(i)}\rightarrow B_{\theta'(i)}(U)\cap_{Y(U)}C^{k-\vartheta'(i)}(U)$
making commutative the diagram (\ref{connection_on_B}). Since the
$\Gamma_k$-connective structure is actually a $(\alpha_{0},\beta_{0};j)$-connection we
have $\alpha_{0,j}\in\mathcal{O}$ and $\beta_{\#,j}\in\mathcal{Q}$.
Thus, for $\theta=\alpha_{0,j}$ and $\vartheta=\beta_{\#,j}$ we
see that for every $\theta'\in\mathcal{O}$, $\vartheta'\in\mathcal{Q}$,
and $i\in[\beta_0;j]_k$ we have a morphism from $B_{\alpha_{0}(j)}(U)\cap_{X(U)}C^{k-\beta_{0}(j)-i}(U)$
to $B_{\theta'(i)}(U)\cap_{Y(U)}C^{k-\vartheta'(i)}(U)$. Notice that
if $f\in B_{\alpha_{0}(j)}(U)\cap_{X(U)}C^{k-\beta_{0}(j)-i}(U)$,
then $\partial^{l}f\in C^{k-\beta_{0}(j)-(i+l)}(U)$ for every $l$
that $i+l\in[\beta_0;j]_k$. Thus, by universality,
for every such $l$ we can regard $\partial^{l}f\in B_{\theta_{l}'(i)}(U)\cap_{Y(U)}C^{k-\vartheta_{l}'(i)}(U)$,
where $\theta'_{l}\in\mathcal{O}$ and $\vartheta'_{l}\in\mathcal{Q}$.
If $z\in[\beta_0;j]_k$, then by definition
we have $z+l\in [\beta_0;j]_k$ for every $l\in \Gamma_k[z]$, so that the functions
$\theta(l)=\theta_{l}'(z)$ and $\vartheta(l)=\vartheta_{l}'(z)$
are defined on $\Gamma_k[z]$. Suppose that this pair $(\theta,\vartheta)$
belongs to $\mathcal{O}\times\mathcal{Q}$, i.e, suppose that it is
ordinary in $\Gamma_k[z]$. Thus, $\partial^{l}f\in B_{\theta(l)}(U)\cap_{Y(U)}C^{k-\vartheta(l)}(U)$,
meaning that $f$ can be regarded as a $(B,k,\theta,\vartheta|\Gamma_k[z])$-function
in $\mathbb{Y}$.
\end{proof}

The Regularity Globalization Lemma will ensure that in each open set $U\subset \mathbb{R}^n$ the local regularity can be globalized. The collage of these conditions when $U$ varies will be made gain, as expected, via partitions of unity. Since we are introducing a new structure (connections between presheaves of $\Gamma$-spaces in an scISP), we need to require some compatibility between them, the partitions of unity and the regular charts of the underlying regular manifold. 

\begin{definition}\label{nice_connections}
Let $B$ be a $B^k_{\alpha,\beta}$-presheaf in a concrete ISP $\mathbb{X}$. We say that a  $(\alpha_{0},\beta_{0};j)$-connection in some scISP $\underline{\mathbb{X}}$ is:
\begin{enumerate}
    \item \emph{Support preserving} if for every pair $(\theta,\vartheta)\in \mathcal{O}\times \mathcal{Q}$, every scISP $\underline{\mathbb{Y}}$ compatible with $\underline{\mathbb{X}}$ and every $U\subset \mathbb{R}^n$, if $f\in B_{\alpha_0(i)}(U)\cap_{X(U)}C^k_b(U)$, then $\operatorname{supp}( \xi^{\alpha_0,\beta_0}_{\theta,\vartheta}(f))\subset \operatorname{supp}(f)$\footnote{Recall that $\alpha_0\in \mathcal{O}$ and $\beta_0\in \mathcal{Q}$, so that the map $\xi^{\alpha_0,\beta_0}$ above exists.}.
    \item \textit{Bump preserving} if for every pair $(\theta,\vartheta)\in \mathcal{O}\times \mathcal{Q}$, every scISP $\underline{\mathbb{Y}}$ compatible with $\underline{\mathbb{X}}$ and every $U\subset \mathbb{R}^n$, the map $\xi^{\alpha_0,\beta_0}_{\theta,\vartheta}$ factors as below\footnote{Compare with Definition \ref{definition_C^k-presheaves}. See Remark \ref{nice_presheaves}.}.$$
\xymatrix{\ar[d]_{\xi^{\alpha_0,\beta_0}_{\theta,\vartheta}} B_{\alpha_0(i)}(U)\cap_{X(U)}C^{k-\beta_0(i)}(U)& B_{\alpha_0(i)}(U)\cap_{X(U)}C^k_b(U) \ar@{_(->}[l] \ar@{-->}[d] \\
B_{\theta(i)}(U)\cap_{Y(U)}C^{k-\vartheta(i)}(U)& B_{\theta(i)}(U)\cap_{X(U)}C^k(U) \ar@{_(->}[l] }
$$
\item \textit{Unital} if given $f\in B_{\alpha_0(i)}(U)\cap_{X(U)}C^{k-\beta_0(i)}(U)$ and $p\in U$ such that $f(p)=1$, then  $(\xi^{\alpha_0,\beta_0}_{\theta,\vartheta}(f))(p)=1$ for every $(\theta,\vartheta)\in \mathcal{O}\times \mathcal{Q}$ and every scISP $\underline{\mathbb{Y}}$ compatible with $\underline{\mathbb{X}}$.
\item \textit{Nice} if it is support preserving, bump preserving and unital.
\end{enumerate}

Next lemma shows that distributive nice $(\alpha_0,\beta_0;j)$-connections preserve partitions of unity.

\begin{lem}\label{lemma_preservation_partition_unity}
Let $B$ be $B^k_{\alpha,\beta}$ in $\mathbb{X}$ and suppose that it admits a distributive and nice $(\alpha_0,\beta_0;j)$-connection in some \emph{scISP} $\underline{\mathbb{X}}$. In this case, if $(U_s)_s$ is a covering of $\mathbb{R}^n$ and $(\psi_s)_s$ is a $C^k$-partition of unity in $\mathbb{R}^n$ subordinate to it, then for every $(\theta,\vartheta)\in \mathcal{O}\times \mathcal{Q}$ the corresponding sequence $(\xi^{\alpha_0,\beta_0}_{\theta,\vartheta}(\psi_s))_s$ is a partition of unity subordinate to the same covering.
\end{lem}
\begin{proof}
Support preserving ensures that each $\xi^{\alpha_0,\beta_0}_{\theta,\vartheta}(\psi_s)$ is a bump function with support in $U_s$, which is $C^k$ due to the bump preserving condition. Support preserving and locally finiteness of $(\psi_s)_s$ imply that $(\xi^{\alpha_0,\beta_0}_{\theta,\vartheta}(\psi_s))_s$ is locally finite too. Furthermore, since the $(\alpha_0,\beta_0;j)$-connection is distributive and unital, we have $$\sum_s \xi^{\alpha_0,\beta_0}_{\theta,\vartheta}(\psi_s)(p)=\xi^{\alpha_0,\beta_0}_{\theta,\vartheta}(\sum_s\psi_s(p))=\xi^{\alpha_0,\beta_0}(1)=1,$$
so that  $(\xi^{\alpha_0,\beta_0}_{\theta,\vartheta}(\psi_s))_s$ is a partition of unity, as desired. 
\end{proof}

\end{definition}
\begin{itemize}
    \item Let $(M,\mathcal{A})$ be a $C^k$-manifold. For every $0\leq r \leq k$, let $\partial^r\mathcal{A}(U)\subset C^{k-r}(U)$ be the set of $r$th derivatives $\partial^r \varphi^a_{ji}$ of transitions functions $\varphi_{ji}=\varphi_j\circ \varphi_i^{-1}$ by charts $\varphi_i,\varphi_j\in \mathcal{A}$ such that $\varphi_i(U_{ij})=U$. Note that if $\mathcal{B}\subset \mathcal{A}$ is a subatlas, then $\partial^r \mathcal{B}(U)\subset\partial^r\mathcal{A}(U)$ for each $r$.  
\end{itemize}
\begin{definition}
Let $(M,\mathcal{B}^k_{\alpha,\beta}(\mathbb{X}))$ be a $(B^k_{\alpha,\beta},\mathbb{X})$-manifold and take $0\leq r\leq k$. We say that a  $(\alpha_{0},\beta_{0};j)$-connection in $B$ in some scISP $\underline{\mathbb{X}}$ has \emph{degree $r$}  if for every par $(\theta,\vartheta)\in \mathcal{O}\times \mathcal{Q}$  such that $i \leq \vartheta(i)$, with $0\leq i\leq r$, every scISP $\underline{\mathbb{Y}}$ compatible with $\mathbb{X}$ and every $U\subset \mathbb{R}^n$, there exists the dotted arrow below (we omitted $U$ to simplify the notation). In other words, $\xi^{\alpha,\beta}_{\theta,\vartheta}(\partial^l \varphi_{ji}) = \partial^l \phi_{ji}$, for every transition function $\varphi_{ji}$ by charts $\varphi_i,\varphi_j\in \mathcal{B}^k_{\alpha,\beta}(\mathbb{X})$, where $\phi_{ji}$ are transition functions of charts $\phi_i,\phi_j\in \mathcal{A}$.
\begin{equation}
    \xymatrix{ \ar@{-->}[d] (B_{\alpha(i)}\cap_{X}C^{k-\beta(i)})\cap_{C^{k-\beta(i)}} \partial^i\mathcal{B}^k_{\alpha,\beta}(\mathbb{X}) \ar@{^(->}[r] & B_{\alpha(i)}\cap_{X}C^{k-\beta(i)} \ar[d]^{\xi^{\alpha,\beta}_{\theta,\vartheta}}  \\
    (B_{\theta(i)}\cap_{X}C^{k-\vartheta(i)})\cap_{C^{k-\vartheta(i)}} \partial^i\mathcal{A} \ar@{^(->}[r] & B_{\theta(i)}\cap_{X}C^{k-\vartheta(i)}}
\end{equation}
\end{definition}
\section{Existence \label{sec_a_b_connections}}

We are now ready to prove the existence theorem of $(B^k_{\alpha_0,\beta_0}|S)$-connections in $B^k_{\alpha,\beta}$-manifolds in the same lines described in the introduction, now presented in a much more precise form:

\begin{enumerate}
     \item \textit{Existence of regular locally defined connections.} In Proposition \ref{prop_local_existence} we will show that every coordinate open set $U\subset M$ of a $C^k$-manifold admits a $B^k_{\alpha_0,\beta_0}$-connection for every $\alpha_0,\beta_0$.
     \item \textit{Existence of weakly locally regular globally defined connections.} In Theorem \ref{thm_existence_connections} we will prove, using the first step and a partition of unity argument, that every $(B^k_{\alpha,\beta},\mathbb{X})$-manifold admits, for every given functions $\alpha_0,\beta_0$, affine connections $\nabla$ whose coefficients in each $(\varphi,U)\in \mathcal{B}^k_{\alpha,\beta}(\mathbb{X})$ belongs to \ref{B_a'_C_b'}, where $\alpha'_0$ and $\beta'_0$ are numbers depending on $\alpha_0$ and $\beta_0$.
     \item \textit{Existence of locally regular globally defined connections.} The connections obtained in the last step are \textit{almost} locally regular because their local coefficients are \textit{not} $(B,k,\alpha',\beta')$-functions for certain $\alpha',\beta'$, but only belong to (\ref{B_a'_C_b'}). In order to be $(B,k,\alpha',\beta')$-functions we also need regularity on the derivative of the coefficients. It is at this moment that Lemma \ref{key_lemma} plays its role and, therefore, the moment when we need to add their hypotheses and additional objects ($(\alpha_0,\beta_0;j)$-connections). This is done in Proposition \ref{prop_a',b'}.
     \item \textit{Existence of globally regular globally defined connections.} Finally, in Theorem \ref{corollary_existence_connections} we glue the locally regular connections in each $(\varphi,U)\in \mathcal{B}^k_{\alpha,\beta}(\mathbb{X})$ getting the desired $B^k_{\alpha',\beta'}$-connections. It is in this step that we will need to assume that the $(\alpha_0,\beta_0;j)$-connections are compatible with partitions of unity, i.e., that they are nice in the sense of Definition \ref{nice_connections}. 
\end{enumerate}

\begin{proposition}[Regular Local Existence]\label{prop_local_existence}
Let $(M,\mathcal{A})$ be a $C^k$-manifold, with $k\geq 2$, and $(\varphi,U)\in\mathcal{A}$ a chart, regarded as a $C^k$-manifold. Let $\Gamma$ be a set, $S \subset [0,k] \subset \Gamma_k$ a subset, and consider functions $\alpha_0:S \rightarrow \Gamma$ and $\beta_0:S \rightarrow [2,k]$. Let $B$ be a presheaf of distributive $\Gamma$-spaces. Then $U$ admits a $(B^k_{\alpha_0,\beta_0}|S)$-connection $\nabla_{\varphi}$ relatively to every concrete ISP between $B_{\alpha_0}$ and $C^{k-\beta_0}$. 
\end{proposition}
\begin{proof}
Recall that any family of $C^{2}$-functions $f_{ab}^{c}:\mathbb{R}^{n}\rightarrow\mathbb{R}$
defines a connection in $\mathbb{R}^{n}$ \cite{chern}. Thus, any
family of $C^{2}$-functions $f_{ab}^{c}:V\rightarrow\mathbb{R}$
in an open set $V\subset\mathbb{R}^{n}$ defines a connection $\nabla_{0}$
in $V$. If $\varphi:U\rightarrow\mathbb{R}^{n}$ is a chart in $M$
with $\varphi(U)\subset V$, pulling back $\nabla_{0}$ we get a connection
$\nabla_{\varphi}$ in $U$ whose coefficients are $f_{ab}^{c}\circ\varphi$.
Consequently, if $f_{ab}^{c}$ are chosen $(B,k,\alpha_{0},\beta_{0}|S)$-functions
in $\mathbb{X}$, then $\nabla_{\varphi}$ is a $(B_{\alpha_0,\beta_0}^{k},\mathbb{X}|S)$-connection
in $U$. 
\end{proof}
\begin{itemize}
    \item If $(*,\epsilon)$ is the multiplicative structure of $B$, let $$\epsilon^{r}(l,m)=\epsilon(l,\epsilon(l,\epsilon(l,\epsilon(...\epsilon(l,m)...))).$$
\end{itemize} 
\begin{theorem}[Locally Weakly Regular Existence]
\label{thm_existence_connections} Let $\mathbb{X}$ be a concrete
ISP and let $M$
be a $(B_{\alpha,\beta}^{k},\mathbb{X})$-manifold, with $k\geq2$. Let $\Gamma_k$ be a set of degree $k$.
Given functions $\alpha_{0}:\Gamma_k \rightarrow \Gamma$ and $\beta_{0}:\Gamma_k \rightarrow [2,k]$, there exists
an affine connection $\nabla$ in $M$ whose coefficients in each $(\varphi,U)\in\mathcal{B}_{\alpha,\beta}^{k}(\mathbb{X})$
belong to 
\begin{equation}
B_{\alpha_{0}'}((\varphi(U'))\cap_{X(\varphi(U'))}C^{k-\beta_{0}'}(\varphi(U'))\label{B_a'_C_b'}
\end{equation}
for some nonempty open set  $U'\subset U$,
\begin{equation}
\alpha_{0}'=\delta(\epsilon^{3}(\alpha(1),\alpha_{0}(0)),\epsilon(\alpha(2),\alpha(1))),\quad\beta_{0}'=\max(\beta(1),\beta(2),\beta_{0}(0))\label{a'_b'}
\end{equation}
and $(\epsilon,\delta)$ are part of the distributive structure of $B$.
\end{theorem}
\begin{proof}
Let $(\varphi_{s},U_{s})_{s}$ be an open
covering of $M$ by coordinate systems in $\mathcal{B}_{\alpha,\beta}^{k}(\mathbb{X})$
and let $\nabla_{s}$ be an affine connection in $U_{s}$ with coefficients
$(f_{s})_{ab}^{c}$. Let $(\psi_{s})$ be a partition of unity of class $C^k$, subordinate
to $(\varphi_{s},U_{s})$, and recall that $\nabla=\sum_{s}\psi_{s}\cdot\nabla_{s}$
is a globally defined connection in $M$. We assert that if $\varphi_{s}$
are in $\mathcal{B}_{\alpha,\beta}^{k}(\mathbb{X})$, then $\nabla$
is a well-defined connection whose cofficients belong to (\ref{B_a'_C_b'})
for certain $U'$. Notice that the coefficients $(\Gamma_{s})_{ab}^{c}$
of $\nabla$ in a fixed $\varphi_{s}$ are obtained in the following
way: let $N(s)$ be the finite set of every $s'$ such that $U_{ss'}\equiv U_{s}\cap U_{s'}\neq\varnothing$.
For each $s'\in N(s)$ we can do a change of coordinates and rewrite
$(f_{s'})_{ab}^{c}$ in the coordinates $\varphi_{s}$ as given by
the functions 
\begin{equation}
(f_{s';s})_{ab}^{c}=\sum_{l,m,o}\frac{\partial\varphi_{s}^{c}}{\partial\varphi_{s'}^{l}}\frac{\partial\varphi_{s'}^{m}}{\partial\varphi_{s}^{a}}\frac{\partial\varphi_{s'}^{o}}{\partial\varphi_{s}^{b}}(f_{s'})_{mo}^{l}+\sum_{l}\frac{\partial^{2}\varphi_{s'}^{l}}{\partial\varphi_{s}^{a}\partial\varphi_{s}^{b}}\frac{\partial\varphi_{s}^{c}}{\partial\varphi_{s'}^{l}},\label{change_coordinates_connection-1}
\end{equation}
where by abuse of notation $\varphi_{s}^{c}=(\varphi_{s}\circ\varphi_{s'}^{-1})^{c}$.
We then have 
\begin{equation}
(\Gamma_{s})_{ab}^{c}=\sum_{s'\in N(s)}\psi_{s'}(f_{s';s})_{ab}^{c}.\label{_partition_unity-1}
\end{equation}
From Proposition \ref{prop_local_existence} we can take each $\nabla_s$ $(B^k_{\alpha_0,\beta_0},\mathbb{X})$-connection, so that $(f_{s'})_{mn}^{l}$ can be choosen $(B,k,\alpha_{0},\beta_{0})$-functions
in $\mathbb{X}$. Furthermore, since $\varphi_{s}\in\mathcal{B}_{\alpha,\beta}^{k}(\mathbb{X})$
we have 
\[
\frac{\partial\varphi_{s}^{c}}{\partial\varphi_{s'}^{l}}\in(B_{\alpha(1)}\cap_{X}C^{k-\beta(1)})(U_{ss'}^{s'})\quad\text{and}\quad\frac{\partial^{2}\varphi_{s'}^{l}}{\varphi_{s}^{a}\partial\varphi_{s}^{b}}\in(B_{\alpha(2)}\cap_{X}C^{k-\beta(2)})(U_{ss'}^{s'})
\]
where $U_{ss'}^{s'}=\varphi_{s'}(U_{ss'})$. This is well-defined
since $\beta(2)\leq k$. Due to the compatibility between the additive/multiplicative
structures of $B_{\alpha}$ and $C^{k-\beta},$ the first and the
second terms of the right-hand sice of (\ref{change_coordinates_connection-1})
belong to
\[
(B_{\epsilon^{3}(\alpha(1),\alpha_{0}(0))}\cap_{X}C^{k-\max(\beta(1),\beta_{0}(0))})(U_{ss'}^{s'})
\]
and
\[
(B_{\epsilon(\alpha(2),\alpha(1))}\cap_{X}C^{k-\max(\beta(2),\beta(1))})(U_{ss'}^{s'}),
\]
respectively, so that the left-hand side of (\ref{change_coordinates_connection-1})
belongs to 
\[
(B_{\tau(\epsilon^{3}(\alpha(1),\alpha_{0}(0)),\epsilon(\alpha(2),\alpha(1)))}\cap_{X}C^{k-\max(\beta(1),\beta(2),\beta_{0}(0))})(U_{N(s)})=(B_{\alpha_{0}'}\cap_{X}C^{k-\beta_{0}'})(U_{N(s)}),
\]
where $U_{N(s)}=\cap_{s'\in N(s)}U_{ss'}$. Since $B$ is a $B_{\alpha,\beta}^{k}$-presheaf
in $\mathbb{X}$, diagram (\ref{diagram_star}) implies that each product $\psi_{s'}(f_{s';s})_{ab}^{c}$ in the sum (\ref{_partition_unity-1}) belongs to the same space, so that the entire sum also belongs to that space. Repeating the process for all coverings in $\mathcal{B}_{\alpha,\beta}^{k}(\mathbb{X})$
and noticing that $\alpha'$ and $\beta'$ do not depend on $s,s'$,
we get the desired result.
\end{proof}
\begin{proposition}[Locally Regular Existence]\label{prop_a',b'} In the same notations and hypotheses of Theorem \ref{thm_existence_connections},
if $B$ admits a $(\alpha',\beta';j)$-connection in some \emph{scISP} $\underline{\mathbb{X}}$, where $\alpha':\Gamma_k \rightarrow \Gamma$ and $\beta':\Gamma_k\rightarrow [2,k]$ are functions such that $\alpha'(j)=\alpha'_0$ and $\beta'(j)=\beta'_0$, then for
every \emph{scISP} $\underline{\mathbb{Y}}$ compatible $\underline{\mathbb{X}}$ and every $z\in [\beta';j]_k$, the elements of (\ref{B_a'_C_b'})
can be regarded as $(B,k,\theta,\vartheta|\Gamma_k[z])$-functions
in $\mathbb{Y}_{\theta,\vartheta}$, for every pair
$(\theta,\vartheta)\in \mathcal{O}\times \mathcal{Q}$ which is ordinary in $\Gamma_k[z]$.
\end{proposition}
\begin{proof}
Direct application of Lemma \ref{key_lemma}.
\end{proof}

\begin{theorem}[Global Regular Existence]
\label{corollary_existence_connections} Let $M$
be a $(B_{\alpha,\beta}^{k},\mathbb{X})$-manifold, with $k\geq2$,
and, as above, suppose that the structural presheaf $B$
admits a \emph{$(\alpha',\beta';j)$}-connection in some \emph{scISP} $\underline{\mathbb{X}}$, where $\alpha'(j)=\alpha'_0$ and $\beta'(j)=\beta'_0$. Suppose, in addition, that this $(\alpha',\beta';j)$-connection is distributive, nice and has
degree $r\geq2$. Then, $M$ admits a $(B_{\theta,\vartheta}^{k},\mathbb{Y}_{\theta,\vartheta}|\Gamma_k[z])$-connection for every \emph{scISP} $\underline{\mathbb{Y}}$ compatible with $\underline{\mathbb{X}}$, every $z\in [\beta';j]_k$ and every pair $(\theta,\vartheta)$ ordinary in $\Gamma_k[z]$ such that $i\leq \vartheta(i)$.
\end{theorem}
\begin{proof}
From Theorem \ref{thm_existence_connections} there exists an affine
connection $\nabla$ whose coefficients (\ref{_partition_unity-1})
belong to (\ref{B_a'_C_b'}). By the first part of the hypotheses we can apply Proposition \ref{prop_a',b'}, so that for any
pair $(\theta,\vartheta)$ ordinary in $\Gamma_k[z]$, the coefficients $(\Gamma_s)^c_{ab}$ in  (\ref{_partition_unity-1}) can be regarded as $(B,k,\theta,\vartheta,\mathbb{Y}_{\theta,\vartheta}|\Gamma_k[z])$-functions. But, from the proof of Lemma \ref{key_lemma} we see that this way to regard elements of an intersection space as regular functions is made by taking their image under $\xi^{\alpha',\beta'}_{\theta,\vartheta}$. Thus, Proposition \ref{prop_a',b'} is telling us precisely that, for every $\varphi_s$, the corresponding $\xi^{\alpha',\beta'}_{\theta,\vartheta}((\Gamma_s)^c_{ab})$ are $(B,k,\theta,\vartheta,\mathbb{Y}_{\theta,\vartheta}|\Gamma_k[z])$-functions. We will show that, under the additional hypotheses, the collection of these when varying $s$ also defines an affine connection in $M$, which will be a $(B_{\theta,\vartheta}^{k},\mathbb{Y}_{\theta,\vartheta}|\Gamma_k[z])$-connection by construction. Notice that the expression (\ref{change_coordinates_connection-1}) which describes how $(\Gamma_s)^c_{ab}$ changes when $s$ varies involves multiplications, sums, bump functions and derivatives of the transition functions. But the additional hypotheses are precisely about the preservation of these data by $\xi$! More precisely, for $i \leq \vartheta(i)$ we have:
\begin{eqnarray*}
\xi^{\alpha',\beta'}_{\theta\vartheta}((f_{s';s})_{ab}^{c}) & = & \xi^{\alpha',\beta'}_{\theta\vartheta}\Bigl(\sum_{l,m,o}\frac{\partial\varphi_{s}^{c}}{\partial\varphi_{s'}^{l}}\frac{\partial\varphi_{s'}^{m}}{\partial\varphi_{s}^{a}}\frac{\partial\varphi_{s'}^{o}}{\partial\varphi_{s}^{b}}(f_{s'})_{mo}^{l}+\sum_{l}\frac{\partial^{2}\varphi_{s'}^{l}}{\partial\varphi_{s}^{a}\partial\varphi_{s}^{b}}\frac{\partial\varphi_{s}^{c}}{\partial\varphi_{s'}^{l}}\Bigr)\\
 & = & \sum_{l,m,o}\xi_{\theta\vartheta}\Bigl(\frac{\partial\varphi_{s}^{c}}{\partial\varphi_{s'}^{l}}\Bigr)\xi_{\theta\vartheta}\Bigl(\frac{\partial\varphi_{s'}^{m}}{\partial\varphi_{s}^{a}}\Bigr)\xi_{\theta\vartheta}\Bigl(\frac{\partial\varphi_{s'}^{o}}{\partial\varphi_{s}^{b}}\Bigr)\xi_{\theta\vartheta}((f_{s'})_{mo}^{l})+ \\
 & & + \sum_{l}\xi_{\theta\vartheta}\Bigl(\frac{\partial^{2}\varphi_{s'}^{l}}{\partial\varphi_{s}^{a}\partial\varphi_{s}^{b}}\Bigr)\xi_{\theta\vartheta}\Bigl(\frac{\partial\varphi_{s}^{c}}{\partial\varphi_{s'}^{l}}\Bigr)\\
 & = & \sum_{l,m,o}\frac{\partial\phi_{s}^{c}}{\partial\phi_{s'}^{l}}\frac{\partial\phi_{s'}^{m}}{\partial\phi_{s}^{a}}\frac{\partial\phi_{s'}^{o}}{\partial\phi_{s}^{b}}\xi^{\alpha',\beta'}_{\theta\vartheta}((f_{s'})_{mo}^{l})+\sum_{l}\frac{\partial^{2}\phi_{s'}^{l}}{\partial\phi_{s}^{a}\partial\phi_{s}^{b}}\frac{\partial\phi_{s}^{c}}{\partial\phi_{s'}^{l}},
\end{eqnarray*}
where we omitted the superior indexes in the maps $\xi$ to  simplify the notation and $\phi_s$ are charts in the maximal $C^k$-atlas $\mathcal{A}$ of $M$.  Furthermore, in the first step we used (\ref{change_coordinates_connection-1}) and
in the second one we used compatibility between the distributive structures
of $B$ and $C^{k-}$ and also the distributive properties of $\xi$ (which comes from the fact that the $(\alpha',\beta';j)$-connection is supposed distributive)  and the composition laws (\ref{composition_law_connections}). In the third step we used that the $(\alpha',\beta';j)$-connection has degree $r\geq2$ (here is where we need $i\leq \vartheta(i)$). By the same kind of arguments,
\begin{eqnarray}
\xi^{\alpha',\beta'}_{\theta,\vartheta}((\Gamma_{s})_{ab}^{c}) & = & \xi^{\alpha',\beta'}_{\theta\vartheta}\Bigl(\sum_{s'\in N(s)}\psi_{s'}\cdot(f_{s';s})_{ab}^{c}\Bigr)\\
 & = & \sum_{s'\in N(s)}\xi_{\theta,\vartheta}^{\alpha',\beta'}(\psi_{s'}) \Bigl(\sum_{l,m,o}\frac{\partial\phi_{s}^{c}}{\partial\phi_{s'}^{l}}\frac{\partial\phi_{s'}^{m}}{\partial\phi_{s}^{a}}\frac{\partial\phi_{s'}^{o}}{\partial\phi_{s}^{b}}\xi^{\alpha',\beta'}_{\theta\vartheta}((f_{s'})_{mo}^{l})+ \\
  & &  \quad \quad \quad \quad \quad \quad  \quad \quad \quad \quad   \quad \quad \quad + \sum_{l}\frac{\partial^{2}\phi_{s'}^{l}}{\partial\phi_{s}^{a}\partial\phi_{s}^{b}}\frac{\partial\phi_{s}^{c}}{\partial\phi_{s'}^{l}}\Bigr) \\
 & = & \sum_{s'\in N(s)}\xi_{\theta,\vartheta}^{\alpha',\beta'}(\psi_{s'})((\Gamma_{s'})_{\theta,\vartheta})^c_{ab} \label{globalization_regular_connection},
\end{eqnarray}
where $((\Gamma_{s'})_{\theta,\vartheta})^c_{ab}$ are the coefficients of the affine connection in $U_{s'}$ defined by the family of functions $\xi^{\alpha',\beta'}_{\theta\vartheta}((f_{s'})_{mo}^{l})$. Since the $(\alpha',\beta';j)$-connection in $B$ is nice and distributive it follows from Lemma \ref{lemma_preservation_partition_unity} that $(\xi^{\alpha',\beta'}_{\theta,\vartheta}(\psi_s))_s$ is a $C^k$-partition of unity for $M$. Thus, (\ref{globalization_regular_connection}) is a globally defined affine connection in $M$, which by construction is a $(B^k_{\theta,\vartheta},\mathbb{Y}_{\theta,\vartheta}|\Gamma_k[z])$-connection.
\end{proof}

\section{Multiplicity \label{sec_multiplicity}}

The last result was about the existence of regular connections. In the following we will consider the multiplicity problem. More precisely, given a $B^k_{\alpha,\beta}$-manifold $M$ in an ISP $\mathbb{X}$, $S\subset \Gamma_k$, functions $\theta:S \rightarrow \Gamma$ and $\vartheta:S\rightarrow [2,k]$, and a concrete ISP $\mathbb{Y}$ compatible with $\mathbb{X}$, we will study the set $\operatorname{Conn}^k_{\theta,\vartheta}(M;\mathbb{Y}|S)$ of $(B_{\theta,\vartheta}^{k},\mathbb{Y}_{\theta,\vartheta}|S)$-connections in $M$. 

\begin{itemize}
    \item Let $(\mathcal{B}_{\operatorname{end}})^k_{\alpha,\beta}(\mathbb{X})\subset \mathcal{B}^k_{\alpha,\beta}(\mathbb{X})$ be the collection of all charts $(\varphi,U)\in \mathcal{B}^k_{\alpha,\beta}(\mathbb{X})$ which trivializes the endomorphism bundle $\operatorname{End}(TM)$.  
\end{itemize} 
\begin{definition}
Let $S\subset \Gamma_k$, $\theta:S\rightarrow \Gamma$, $\vartheta:S\rightarrow [2,k]$ and $\mathbb{Y}$ as above. An $\operatorname{End}(TM)$-valued $(B^k_{\theta,\vartheta},\mathbb{Y})$-form of degree one is an $\operatorname{End}(TM)$-valued $1$-form $\omega$ of class $C^k$ whose coefficients $(\omega_\varphi)^c_{ab}$ in every $(\varphi,U)\in (\mathcal{B}_{\operatorname{end}})^k_{\alpha,\beta}(\mathbb{X})$ are $(B,k,\theta,\vartheta|S)$-functions in $\mathbb{Y}$. Let 
\begin{equation}\label{regular_1_forms}
\Omega^1_{\theta,\vartheta}(M;\operatorname{End}(TM),\mathbb{Y}|S)    
\end{equation} denote the set of all of them.
\end{definition}

We begin with an easy fact: in the same way as the set of affine connections in a $C^k$-manifold is an affine space of the vector space of $\operatorname{End}(TM)$-valued 1-forms, the set of regular affine connections on a regular manifold is an affine space of the vector space of regular  $\operatorname{End}(TM)$-valued 1-forms. More precisely,
\begin{proposition}\label{space_affine_connections_is_affine}
Let $M$ be a $(B^k_{\alpha,\beta},\mathbb{X})$-manifold. For every $S$, $\theta$, $\vartheta$ and $\mathbb{Y}$ as above:
\begin{enumerate}
     \item \emph{(\ref{regular_1_forms})} is a vector subspace of $\Omega^1(M;\operatorname{End}(TM))$;
    \item $\operatorname{Conn}^k_{\theta,\vartheta}(M;\mathbb{Y}|S)$ is an affine space of \emph{(\ref{regular_1_forms})}.
\end{enumerate}
 
\end{proposition}
\begin{proof}
For the first one, recall that in
a distributive $\Gamma$-space $B=(B_{i})$ the additive structure
$+_{ij}:B_{i}\otimes B_{j}\rightarrow B_{\delta(i,j)}$ is such that
$\delta(i,i)=i$ and $+_{ii}$ is the sum of $B_{i}$. For the second one, if $\nabla$ and $\overline{\nabla}$ are two affine $C^k$-connections
in $M$, then $\omega=\nabla-\overline{\nabla}$ has coefficients in $(\varphi,U)\in (\mathcal{B}_{\operatorname{end}})^k_{\alpha,\beta}(\mathbb{X})$ given by $(\omega_\varphi)_{ab}^{c}=(\Gamma_\varphi)_{ab}^{c}-(\overline{\Gamma}_\varphi)_{ab}^{c}$. Thus, again by $+_{ii}=+_i$, if $(\Gamma_\varphi)_{ab}^{c}$ and  $(\overline{\Gamma}_\varphi)_{ab}^{c}$ are $(B,k,\theta,\vartheta,\mathbb{Y}|S)$-functions, then $(\omega_\varphi)_{ab}^{c}$ are so.
\end{proof}

We will now prove that  $(B_{\theta\vartheta}^{k},\mathbb{Y}|S)$-connections $\nabla$ and $\overline{\nabla}$ which are ``locally additively different'' are actually
``locally different''.
\begin{definition}
Let $M$ be a $(B^{k}_{\alpha,\beta},\mathbb{X})$-manifold. Given $S\subset \Gamma_k$, $\theta$, $\delta$ and $\mathbb{Y}$ as above, a \textit{3-parameter $(\theta,\vartheta,\mathbb{Y}|S)$-family} in $M$ is just a collection $(\Omega_\varphi)^{c}_{ab}$ of $(B,k,\theta,\vartheta|S)$-functions in $\mathbb{Y}$, with $a,b,c=1,...,n$, for each $(\varphi,U)\in \mathcal{B}^k_{\alpha,\beta}(\mathbb{X})$.
\end{definition}
\begin{itemize}
    \item If $(\Omega_\varphi)^c_{ab}$ is a 3-parameter family, let
    $
    (\Omega_\varphi)^c= \sum_{a,b}(\overline{\Omega}_\varphi)^{c}_{ab}(p)
    $.
\end{itemize}
\begin{definition}
Two 3-parameter $(\theta,\vartheta,\mathbb{Y}|S)$-families $(\Omega_\varphi)^{c}_{ab}$ and $(\overline{\Omega}_\varphi)^{c}_{ab}$ are:
\begin{enumerate}
    \item \textit{locally additively different} if for every chart $(\varphi,U) \in \mathcal{B}^k_{\alpha,\beta}(\mathbb{X})$ and every $c=1,...,n$ we have $(\Omega_\varphi)^c\neq (\overline{\Omega}_\varphi)^c$;
    \item \textit{locally different} if for every chart $(\varphi,U) \in \mathcal{B}^k_{\alpha,\beta}(\mathbb{X})$ and every $a,b,c=1,...,n$ there exists a nonempty open set $U_{a,b,c}\subset U$ such that $(\Omega_\varphi)^c_{ab}(p)\neq (\overline{\Omega}_\varphi)^c_{ab}(p)$ for every $p\in U_{a,b,c}$.
    
\end{enumerate}

\end{definition}
\begin{definition}
Two $(B^k_{\theta,\vartheta},\mathbb{Y}|S)$-connections $\nabla$ and $\overline{\nabla}$ in $M$ are \textit{locally additively different} (resp. \textit{locally different}) if the corresponding  3-parameter $(\theta,\vartheta,\mathbb{Y}|S)$- families of their coefficients are locally additively different (resp. locally different).
\end{definition}
\begin{theorem}\label{multiplicity_thm}
If two 3-parameter $(\theta,\vartheta,\mathbb{Y}|S)$- families $(\Omega_\varphi)^{c}_{ab}$ and $(\overline{\Omega}_\varphi)^{c}_{ab}$ in a smooth $(B^{\infty}_{\alpha,\beta},\mathbb{X})$-manifold are locally additively different, then they are locally different.
\end{theorem}
\begin{proof}
 For every chart $(\varphi,U)$ and every $c=1,..,n$, let $U_c\subset U$ be the subset in which $(\Omega_{\varphi})^c(p)\neq (\overline{\Omega}_\varphi)^c(p)$.  Since the 3-parameters are locally additively different, the subsets $U_c$ are nonempty. On the other hand, they are the complement of the closed sets $[(\Omega_{\varphi})^c-(\overline{\Omega}_\varphi)^c]^{-1}(0)$. Thus, they are nonempty of sets. For each $p\in U_c$ we then have
\[
0<\vert\Omega_{\varphi}^{c}(p) - (\overline{\Omega}_\varphi)^c(p) \vert\leq\sum_{a,b}\vert(\Omega_{\varphi}){}_{ab}^{c}(p)-(\overline{\Omega}_{\varphi}){}_{ab}^{c}(p)\vert,
\]
implying the existence of functions $a_c,b_c:U_c \rightarrow[n]$
such that 
\[
(\Omega_{\varphi}){}_{a_c(p)b_c(p)}^{c}(p)\neq(\overline{\Omega}_{\varphi}){}_{a_c(p)b_c(p)}^{c}(p).
\]
Regard $[n]$ as a $0$-manifold and notice that if two smooth functions
$f,g:V\rightarrow[n]$ are transversal, then they coincide in an open
set $V'\subset V$. Since $[n]$ is a $0$-manifold, any smooth map
$f:V\rightarrow[n]$ is a submersion and therefore transversal to
each other. In particular, $a_c$ (resp. $b_c$)
is transversal to each constant map $a'\in[n]$ (resp. $b'$), so
that for each $a'$ (resp. $b'$) there exists $U_{a',c}\subset U_c$
(resp. $U_{b',c}\subset U_c$) in which 
\[
(\Omega_{\varphi}){}_{a'b_c(p)}^{c}(p)\neq(\overline{\Omega}_{\varphi}){}_{a'b_c(p)}^{c}(p)\quad(\text{resp.}\quad(\Omega_{\varphi}){}_{a_c(p)b'}^{c}(p)\neq(\overline{\Omega}_{\varphi}){}_{a_c(p)b'}^{c}(p))
\]
Taking $U_{a',b',c}=U_{a',c}\cap U_{b',c}$ and noticing that this intersection
is nonempty, the proof is done.
\end{proof}
\subsection{Discussion}\label{sec_discussion}
We close with a broad discussion which we would like to see formalized and enlarged in some future work.
\begin{itemize}
    \item Let $\mathcal{C}_k(\theta,\vartheta;\mathbb{Y})$ denote the set of 3-parameter $(\theta,\vartheta,\mathbb{Y})$-families in some  $(B^k_{\alpha,\beta},\mathbb{X})$-manifold $M$. Let $\xi$ be a nice and distributive $(\alpha',\beta';j)$-connection in $M$ with degree $r\geq0$.
\end{itemize}

The construction of Section \ref{sec_a_b_connections} shows that every 3-parameter $(\alpha_0,\beta_0,\mathbb{X})$-family $f=(f_\varphi)^c_{ab}$ in $M$  induces, for each ordinary pair $(\theta,\vartheta)$, a corresponding  $(B^k_{\theta,\vartheta},\mathbb{Y}_{\theta,\vartheta})$-connection $\nabla^{f}_{\xi}$ in $M$, whose 3-parameter family of coefficients $(\Gamma^f_{\xi, \varphi})^c_{ab}$ is given by (\ref{globalization_regular_connection}). Suppose, for a moment, that we found a class $\mathcal{S}\subset \mathcal{C}^k_3(\alpha_0,\beta_0;\mathbb{X})$ of 3-parameter $(\alpha_0,\beta_0,\mathbb{X})$-families such that for every $f\in \mathcal{S}$ the corresponding regular connection $\nabla^f_{\xi}$ has coefficients $(\Gamma^f_{\xi,\varphi})^c_{ab} \in \mathcal{C}_k(\theta,\vartheta;\mathbb{Y}) $ which are locally additively different to each other $\Omega_\varphi\in \mathcal{C}_k(\theta,\vartheta;\mathbb{Y})$. Thus, by Theorem \ref{multiplicity_thm} it follows that they are also locally different. In particular, \textit{every 3-parameter $(\theta,\vartheta,\mathbb{Y})$-family is locally additively different to some $(B^k_{\theta,\vartheta},\mathbb{Y})$-connection and this ``difference'' is measured by elements in $\mathcal{S}$.} This is the typical shape of denseness theorems. More precisely, one could expect the existence of a topology in $\mathcal{C}_k(\theta,\vartheta;\mathbb{Y})$ whose basic open sets are parameterized by elements of $\mathcal{S}$ and such that $\operatorname{Conn}^k_{\theta,\vartheta}(M;\mathbb{Y})\subset \mathcal{C}_k(\theta,\vartheta;\mathbb{Y})$ is a dense subset.

Actually, the existence of $\mathcal{S}$ as above also implies that \textit{if $f\in\mathcal{S}$, then $\nabla^f _{\xi}$ is locally additively different to each other $\nabla\in\operatorname{Conn}^k_{\theta,\vartheta}(M;\mathbb{Y})$}, which is the shape of an openness theorem. Thus, we could expect that, if the topology described above exists, then $\operatorname{Conn}^k_{\theta,\vartheta}(M;\mathbb{Y})\subset \mathcal{C}_k(\theta,\vartheta;\mathbb{Y})$ should be open and nonempty due Theorem \ref{corollary_existence_connections}. 

Finally, looking again at the expressions  (\ref{change_coordinates_connection-1}) and (\ref{globalization_regular_connection}) defining $\nabla ^f_{\xi}$, one concludes that the problem of finding $\mathcal{S}$ consists in studying the complement in $\mathcal{C}_k(\theta,\vartheta;\mathbb{Y})$ of the solution spaces $\mathcal{X}_{\xi}(\varphi,\Omega)$, for each $\varphi\in \mathcal{B}^{k}_{\alpha,\beta}(\mathbb{X})$ and each $\Omega\in \mathcal{C}_k(\theta,\vartheta;\mathbb{Y})$, of the system of nonhomogeneous equations 

\begin{equation}\label{final_equation}
 F^c_{\xi}(f;\varphi,\Omega)=(\Omega_\varphi)^c,\quad \text{where}\quad F^c_{\xi}(f;\varphi,\Omega)= \sum_{a,b} \xi ^{\alpha',\beta'}_{\theta,\vartheta}((\Gamma ^f_{\xi,\varphi})^c_{ab}),   
\end{equation} 
whose coefficients belong to the space of $(B,k,\theta,\vartheta|\Gamma_k[z])$-functions. Since  equations (\ref{final_equation}) strongly depend on the $(\alpha',\beta';j)$-connection $\xi$, the basic strategy to find $\mathcal{S}$ (and then to get the openness and denseness results described above) would be to introduce additional conditions in $\xi$ which simplify (\ref{final_equation}). In other words, would be to look at (\ref{final_equation}) as parameterized systems of equations in the space of all $(\alpha',\beta';j)$-connections.

\section*{Acknowledgments}

The authors thank the reviewer's valuable considerations, which undoubtedly made the text much more readable. The first author was supported by CAPES (grant number 88887.187703/2018-00).

\end{document}